\documentclass[11pt]{amsart}

\usepackage{pgf,tikz}
\usetikzlibrary{arrows}

\usepackage{amsthm}
\usepackage{amssymb}
\usepackage{amsmath}
\usepackage[shortlabels]{enumitem}

\usepackage{ifpdf}
\ifpdf
\usepackage[pdftex,linktocpage]{hyperref}
\else
\usepackage[hypertex,linktocpage]{hyperref}
\fi

\input xy
\xyoption{all} 

\newtheorem{theorem}{Theorem}
\newtheorem*{theorem*}{Theorem}

\newtheorem{example}[theorem]{Example}
\newtheorem{definition}[theorem]{Definition}
\newtheorem*{definition*}{Definition}
\newtheorem{lemma}[theorem]{Lemma}

\newtheorem{proposition}[theorem]{Proposition}
\newtheorem{remark}[theorem]{Remark}
\newtheorem{cor}[theorem]{Corollary}
\newtheorem*{cor*}{Corollary}

\newtheorem{question}{Question}

\newtheorem*{conjecture*}{Conjecture}
\numberwithin{theorem}{section}

\renewcommand\iff{%
\ifmmode\text{ if and only if }%
\else if and only if \fi}

\renewcommand{\and}{\wedge}

\renewcommand{\phi}{\varphi}

\newcommand{\elsoc}{\text{elsoc}}

\newcommand{\Mod}{\textnormal{Mod}}
\renewcommand{\mod}{\text{mod}}

\newcommand{\Ab}{\text{Ab}}

\newcommand{\rad}{\textnormal{rad}}

\newcommand{\D}{\textnormal{D}}

\newcommand{\End}{\textnormal{End}}
\newcommand{\Hom}{\textnormal{Hom}}

\newcommand{\Zg}{\textnormal{Zg}}

\newcommand{\pinj}{\textnormal{pinj}}

\newcommand{\im}{\textnormal{im}}

\newcommand{\mcal}[1]{\mathcal{#1}}

\newcommand{\mfrak}[1]{\mathfrak{#1}}
\newcommand{\st}{\ \vert \ }

\newcommand{\pp}{\textnormal{pp}}

\newcommand{\N}{\mathbb{N}}
\newcommand{\Z}{\mathbb{Z}}

\newcommand{\coker}{\text{coker}}

\newenvironment{pmat}{\left( \begin{smallmatrix}}{\end{smallmatrix} \right)}

\newcommand{\downmapsto}{\rotatebox[origin=c]{-90}{$\scriptstyle\longmapsto$}\mkern2mu}

\title{Ziegler closures of some unstable tubes}
\author{Lorna Gregory}
\thanks{The author acknowledges the support of EPSRC through Grant EP/K022490/1 and thanks Ivo Herzog and Mike Prest for useful discussions. This work was completed while the author was at the University of Manchester.}
\address{University of Camerino, School of Science and Technologies, Division
of Mathematics, Via Madonna delle Carceri 9, 62032 Camerino, Italy }
\email{lorna.gregory@gmail.com}
\date{\today}
\begin{document}
\maketitle

\begin{abstract}
We describe the modules in the Ziegler closure of ray and coray tubes in module categories over finite-dimensional algebras. We improve slightly on Krause's result for stable tubes by showing that the inverse limit along a coray in a ray or coray tube is indecomposable, so in particular, the inverse limit along a coray in a stable tube is indecomposable. In order to do all this we first describe the finitely presented modules over and the Ziegler spectra of iterated one-point extensions of valuation domains. Finally we give a sufficient condition for the $k$-dual of a $\Sigma$-pure-injective module over a $k$-algebra to be indecomposable.
\end{abstract}

Ray and coray tubes frequently occur in Auslander-Reiten quivers of finite-dimensional algebras. After stable tubes (also referred to as smooth tubes) they are the simplest form of Ringel and D'Este's coherent tubes.

Generalising Krause's definition, \cite[pg 20]{Krausegenmods}, of a generalised tube we introduce the notions of a generalised coray tube and a generalised ray tube. Our aim here is to use these notions to describe the Ziegler closures of unstable tubes containing no projective modules (dually, unstable tubes containing no injective modules).

As in the case of stable tubes, we show that every ray tube has finitely many non-finitely presented indecomposable pure-injectives in its Ziegler closure each of which is either a direct limit along a ray, an inverse limit along a coray or a generic module. Improving slightly on Krause's results we show that in any ray or coray tube (thus also in any stable tube) the inverse limit along a coray is indecomposable. This result was claimed in \cite[15.1.10]{PSL} but no proof is given and the proof indicated there does not work.

In section \ref{iteratedonepointext}, we describe the finitely presented modules over iterated one-point extensions of discrete valuation domains. This allows us, see section \ref{Indpis}, to describe the indecomposable pure-injectives over iterated one-point extensions of discrete valuation domains.

Using these results, we describe a functor from the module category of the iterated one-point extension of $k[[x]]$ to a module category over a finite-dimensional $k$-algebra containing a generalised ray tube such that every module in the closure of the generalised ray tube is a direct summand of some module in the image of this functor.

In section \ref{shortembeddings} we introduce short embeddings. These are embeddings $f:M\rightarrow L$ between finite-length modules over an artin algebra such that for all $a\in M$, if $\phi$ generates the pp-type of $a$ in $M$ and $\psi$ generates the pp-type of $f(a)$ in $L$ then the interval $[\psi,\phi]$ in the pp-$1$-lattice is finite-length. Equivalently, an embedding is short if the cokernel of $(f,-):(L,-)\rightarrow (M,-)$ in $(\mod-R,\Ab)^{\text{fp}}$ is finite-length.

Short embeddings allow us to investigate the endomorphism rings of direct limits along rays in ray and coray tubes.  Using results in \ref{shortembeddings}, we are able to show that direct limits along rays in ray and coray tubes are indecomposable and moreover that the canonical embedding of $k$ into their endomorphism rings factored out by the radical is an isomorphism. In section \ref{elsocanddual} we will show that this implies that their $k$-duals are indecomposable.

In the final section we put all this together to describe the Ziegler closures of ray and coray tubes for finite-dimensional algebras. We end by discussing some open questions.

Throughout this paper we switch freely between the more concrete pp-formula formalism and the more abstract functor category formalism.

Through out this paper, if $R$ is a ring then $\Mod-R$ (respectively $R-\Mod$) denotes the category of right $R$-modules (respectively left $R$-modules) and $\mod-R$ (respectively $R-\mod$) denotes that category of finitely presented right $R$-modules (respectively left $R$-modules).

\section{Preliminaries}
Let $R$ be a ring.

A \textbf{pp-$n$-formula} is a formula in the language $\mcal{L}_R=(0,+,(\cdot r)_{r\in R})$ of (right) $R$-modules of the form
\[\exists \overline{y} (\overline{x},\overline{y})H=0\] where $\overline{x}$ is a $n$-tuple of variables and $H$ is an appropriately sized matrix with entries in $R$. If $\phi$ is a pp-formula and $M$ is a right $R$-module then $\phi(M)$ denotes the set of all elements $\overline{m}\in M^n$ such that $\phi(\overline{m})$ holds. Note that for any module $M$, $\phi(M)$ is a subgroup of $M^n$ equipped with the addition induced by addition in $M$. A pair of pp-$n$-formulas $\phi/\psi$ is called a \textbf{pp-pair} if for all $R$-modules $M$, $\phi(M)\supseteq\psi(M)$.

If we weaken our definition of a pp-formula to include all formulas (in one variable) in the language of (right) $R$-modules, $\mcal{L}_R$, which are equivalent over the theory of $R$-modules, $T_R$, to a pp-formula then the $T_R$-equivalence classes of pp-$n$-formulas become a lattice under implication with the join of two formulas $\phi,\psi$ given by
\[(\phi+\psi)(x):=\exists y,z(x=y+z\and\phi(y)\and\psi(z))\] and the meet given by $\phi\wedge\psi$. Given a (right) $R$-module $M$, we write $\pp^1_R(M)$ for the lattice of pp-definable subgroups of $M$.

Let $M$ be an $R$-modules and $\overline{m}\in M^n$. The \textbf{pp-type} of $\overline{m}$ in $M$, denoted $\pp_R^M(\overline{m})$, is the set all pp-$n$-formulas $\phi$ such that $m\in \phi(M)$.

If $M$ is finitely presented module and $\overline{m}\in M^n$ then there is a pp-$n$-formula $\phi$ which \textbf{generates the pp-type} of $\overline{m}$ in $M$, that is, for all pp-formulas $\psi$, $\psi\geq\phi$ if and only if $\overline{m}\in\psi(M)$. Conversely, if $\phi$ is a pp-$n$-formula, then there exists a finitely presented module $M$ and $\overline{m}\in M^n$ such that $\phi$ generates the pp-type of $\overline{m}$ in $M$. See \cite[Section 1.2.2]{PSL} for proofs.

Any pp-pair $\phi/\psi$, gives rise to a finitely presented functor $F_{\phi/\psi}:\mod-R\rightarrow \Ab$ which sends a right $R$-module $M$ to $\phi(M)/\psi(M)$ and conversely, any finitely presented functor is isomorphic to one of the form $F_{\phi/\psi}$, \cite[10.2.30]{PSL}.

Given a lattice $L$, let $\sim$ be the congruence relation generated by the simple intervals in $L$. We define what it means for a lattice to have finite m-dimension. For the more general ordinal valued dimension, see \cite[section 7.2]{PSL}. We say that $L$ has \textbf{m-dimension} $0$ if $L/\sim$ is the one point lattice. We say that $L$ has \textbf{m-dimension} $n+1$ if $L/\sim$ has m-dimension $L$.

We say that a module has m-dimension $n$ if its lattice of pp-definable subgroups has m-dimension $n$.

An embedding of $R$-modules $f:M\rightarrow N$ is \textbf{pure} if for every pp-$1$-formula $\phi$, $f(\phi(M))=\phi(N)\cap f(M)$. A right $R$-module $M$ is \textbf{pure-injective} if it is injective over all pure-embeddings.

In section \ref{elsocanddual}, we will use that a module is pure-injective if and only if it is algebraically compact \cite[4.3.11]{PSL}. An $R$-module $M$ is said to be \textbf{algebraically compact} if every system of equations over $R$ in arbitrary many variables with parameters in $M$ such that every finite subsystem has a solution in $M$, has a solution in $M$. Equivalently, a module is algebraically compact if every collection of cosets of pp-definable subgroups which has the finite intersection property has non-empty intersection.

We say that a module is \textbf{$\Sigma$-pure-injective} if it has no descending chain of pp-definable subgroups. Every $\Sigma$-pure-injective module is pure-injective.

The  (right) \textbf{Ziegler spectrum} of a ring $R$ is a topological space with set of points, $\pinj_R$, the isomorphism classes of indecomposable pure-injectives and basis of open sets given by
\[\left(\phi/\psi\right):=\{N\in\pinj_R \st \phi(N)\neq \psi(N)\}\] where $\phi/\psi$ is a pp-pair.

Our descriptions of Ziegler closures and Ziegler spectras in this paper are all based on the Ziegler spectrum of a discrete valuation domain $V$.

The indecomposable pure-injectives over $V$ are the indecomposable finite-length modules $V/\mfrak{m}^l$ for $l\in\N$, the Pr\"{u}fer module $E(V/\mfrak{m})$ (that is, the injective hull of $V/\mfrak{m}$), the completion $\widehat{V}$ of $V$ and $Q(V)$ the field of fractions of $V$. See for instance \cite[5.1]{Zieglermodules}.

A subset $X$ of $\Zg_V$ is closed if the following two properties hold:
\begin{enumerate}
\item If $X$ contains infinitely many finite-length modules then $X$ contains $E(V/\mfrak{m})$, $\widehat{V}$ and $Q(V)$.
\item If $X$ contains $\widehat{V}$ or $E(V/\mfrak{m})$ then $X$ contains $Q(V)$.
\end{enumerate}

This description of the topology can be extracted from \cite[section 5.2.1]{PSL}.

A definable subcategory of $\Mod-R$ is a full subcategory closed under direct limits, products and pure-submodules. If $\mcal{D}$ is a definable subcategory then there exists pp-pairs $\phi_\lambda/\psi_\lambda$, $\lambda\in I$ such that $M\in\mcal{D}$ if and only if $\phi_\lambda(M)=\psi_\lambda(M)$ for all $\lambda\in I$ and conversely, all subcategories of this form are definable subcategories.

A definable subcategory of $\Mod-R$ is determined by the indecomposable pure-injective it contains \cite[5.1.4]{PSL}. Thus there is a bijective correspondence between closed subset of $\Zg_R$ and definable subcategories of $\Mod-R$.

\begin{definition}\label{dualitydef}
Let $\phi$ be a pp-$n$-formula in the language of right $R$-modules of the form $\exists \bar{y}(\bar{x},\bar{y})H=0$ where $\bar{x}$ is a tuple of $n$ variables, $\bar{y}$ is a tuple of $l$ variables, $H=(H' \ H'')^T$ and $H'$ (respectively $H''$) is a $n\times m$ (respectively $l\times m$) matrix with entries in $R$. Then $\D\phi$ is the pp-$n$-formula in the language of left $R$-modules $\exists \bar{z}\left(
                                                       \begin{array}{cc}
                                                         I & H' \\
                                                         0 & H'' \\
                                                       \end{array}
                                                     \right)\left(
                                                              \begin{array}{c}
                                                                \bar{x} \\
                                                                \bar{z} \\
                                                              \end{array}
                                                            \right)=0
$.

Similarly, let $\phi$ be a pp-$n$-formula in the language of left $R$-modules of the form $\exists \bar{y}H\left(
                                                                                                             \begin{array}{c}
                                                                                                               \bar{x} \\
                                                                                                               \bar{y} \\
                                                                                                             \end{array}
                                                                                                           \right)
=0$ where $\bar{x}$ is a tuple of $n$ variables, $\bar{y}$ is a tuple of $l$ variables, $H=(H' \ H'')$ and $H'$ (respectively $H''$) is a $m\times n$ (respectively $m\times l$) matrix with entries in $R$. Then $\D\phi$ is the pp-$n$-formula in the language of right $R$-modules $\exists \bar{z}(\bar{x},\bar{z})\left(
                                                       \begin{array}{cc}
                                                         I & 0 \\
                                                         H' & H'' \\
                                                       \end{array}
                                                     \right)=0
$.
\end{definition}

Note that the pp-formula $a|x$ for $a\in R$ is mapped by $D$ to a formula equivalent with respect to $T_R$ to $xa=0$ and the pp-formula $xa=0$ for $a\in R$ is mapped by $D$ to a formula equivalent with respect to $T_R$ to $a|x$.

\begin{theorem}\cite[Chapter 8]{Mikebook1}\cite[1.3.1]{PSL}
The map $\phi\rightarrow D\phi$ induces an anti-isomorphism between the lattice of right pp-$n$-formulae and the lattice of left pp-$n$-formulae. In particular, if $\phi,\psi$ are pp-$n$-formulae then $D(\phi+\psi)$ is equivalent to $D\phi\wedge D\psi$ and $D(\phi\wedge\psi)$ is equivalent to $D\phi+D\psi$.
\end{theorem}

Let $R$ be a $k$-algebra. We denote that standard dual $\Hom(-,k)$ of an $R$-module $M$, respectively a morphism $f$, by $M^*$, respectively $f^*$.
\begin{lemma}\label{kdualandduality}
Let $R$ be a $k$-algebra and $M$ an $R$-module. If $\phi(M)\leq \psi(M)$ then $D\psi(M^*)\leq D\phi(M^*)$.
\end{lemma}

Thus if $M$ is an $R$-module with dcc (respectively acc) on pp-definable subgroups then $M^*$ has acc (respectively dcc) on pp-definable subgroups. So, in particular, if the pp-$1$-lattice of $M$ is finite-length then so is the pp-$1$-lattice of $M^*$.

Another key tool in this paper is the use of additive functors $F:\Mod-R\rightarrow \Mod-S$ which commute with direct limits and arbitrary products. Such functors are called \textbf{interpretation functors}. These are exactly the functors which are finitely presented when composed with the forgetful functor to $\Ab$ and thus given by pp-pairs. This is where there name comes from \cite{intmodulesinmodules}. See \cite[7.2]{exdefcats} and \cite[12.9]{Defaddcats} for the equivalence.

We will need the following facts about interpretation functors $F:\Mod-R\rightarrow \Mod-S$.

\begin{enumerate}
\item There exists an $n\in\N$, such that for all $M\in\Mod-R$, there is a lattice embedding $\pp_S^1(FM)\hookrightarrow \pp_R^n(M)$.
\item If $M$ is pure-injective then $FM$ is pure-injective. See \cite[3.16]{intmodulesinmodules} or \cite[6.1]{exdefcats}.
\item If $M$ is $\Sigma$-pure-injective (respectively has the acc on pp-definable subgroups) then $FM$ is $\Sigma$-pure-injective (respectively has the acc on pp-definable subgroups). Follows from (1).
\item If $\pp_R^1(M)$ has m-dimension $\alpha$ then $\pp_S^1(FM)$ has m-dimension less than or equal to $\alpha$. Follows from (1) plus the fact that the m-dimension of $\pp_R^1(M)$ is equal to the m-dimension of $\pp_R^n(M)$.
\item If $\mcal{D}$ is a definable subcategory of $\Mod-R$ then after closing under direct summands and isomorphism, $F\mcal{D}$ is a definable subcategory. See \cite[3.8]{Abcatsanddefaddcats}.
\item If the m-dimension of $\pp_R^1$ is $\alpha$ then the m-dimension of the smallest definable subcategory containing the image of $F$ is less that or equal to $\alpha$.
\end{enumerate}

\section{Iterated one-point extensions of discrete valuation domains}\label{iteratedonepointext}

Let $R$ be a ring, $k$ a field and $L$ a $k-R$-bimodule. The one-point extension of $R$ by $L$ is the ring
\[R[L]:=\left(
    \begin{array}{cc}
      R & 0 \\
      L & k \\
    \end{array}
  \right).
\]

Right modules over $R[L]$ may be viewed as triples $(M_0,M_1,\Gamma_M)$ where $M_0$ is a $k$-vector space, $M_1$ is a right $R$-module and $\Gamma_M:M_0\rightarrow \Hom_R(L,M_1)$ is a $k$-homomorphism. A morphism between two triples $(N_0,N_1,\Gamma_N)$ and $(M_0,M_1,\Gamma_M)$ is given by a pair $(f_0,f_1)$ where $f_0:N_0\rightarrow M_0$ is a $k$-vector space homomorphism, $f_1:N_1\rightarrow M_1$ is an $R$-module homomorphism and the following diagram commutes
\[\xymatrix@R=20pt@C=80pt{
  N_0 \ar[d]_{f_0} \ar[r]^{\Gamma_N} & \Hom_R(L,N_1) \ar[d]^{\Hom(L,f_1)} \\
  M_0 \ar[r]^{\Gamma_M} & \Hom_R(L,M_1)   }.\]

There are two full and faithful embeddings of $\Mod-R$ into $\Mod-R[L]$:
\[F_0:\Mod-R\rightarrow \Mod-R[L] \ \ M\mapsto (0,M,0)\]
and
\[F_1:\Mod-R\rightarrow \Mod-R[L] \ \ M\mapsto (\Hom(L,M),M, \text{Id}_{\Hom(L,M)}).\]

The forgetful functor
\[r:\Mod-R[L]\rightarrow R \ \ (M_0,M_1,\Gamma)\mapsto M_1\] is right adjoint to $F_0$ and left adjoint to $F_1$.

Each of these additive functors commute with direct limits and arbitrary products and thus are interpretation functors.

Throughout the rest of this section, let $V$ be a discrete valuation domain with maximal ideal $\mfrak{m}$ and let $R_0:=V$, $L_0=V/\mfrak{m}=:k$,

\[R_{n+1}:=\left(
             \begin{array}{cc}
               R_n & 0 \\
               L_n & k \\
             \end{array}
           \right)
\] and $L_{n+1}=F_1L_n$.

Finally, for each $n\in \N$, let $T_n:=(k,0,0)\in\Mod-R_n$.

The category of finitely presented modules over a discrete valuation domain $V$ is Krull-Schmidt and the indecomposable finitely presented modules are $V$ and  $V/\mfrak{m}^n$ where $n\in\N$.

Our goal for the rest of this section is to prove the following theorem and thus classify the finitely presented right $R_n$-modules.

\begin{theorem}\label{descfpind}
The category of finitely presented right modules over $R_n$ is Krull-Schmidt. The indecomposable finitely presented modules over $R_n$ are of the form $T_n$, $F_0^{m}F_1^{n-m}N$ where $0\leq m\leq n$ and $N$ is an indecomposable finitely presented module over $V$ and $F_0^{n-k-l}F_1^{l}T_k$ where $k+l\leq n$ and $0\leq k,l$.
\end{theorem}

\begin{lemma}\label{fpfp}
If $M=(M_0,M_1,\Gamma)$ is finitely presented then $M_1$ is finitely presented and $M_0$ is finite-dimensional.
\end{lemma}
\begin{proof}
Suppose $R_{n+1}^l\xrightarrow{f} R_{n+1}^m\rightarrow M\rightarrow 0$ is a presentation for $M$. Applying the exact functor $r$ we get that $rR_{n+1}^l\xrightarrow{rf} rR_{n+1}^m\rightarrow M_1\rightarrow 0$ is a presentation for $M_1$. Since $rR_{n+1}$ is finitely presented, $M_1$ is finitely presented.

As a module over itself $R_{n+1}$ is $(k,R_n\oplus L_n, \Gamma)$ where $\Gamma:k\rightarrow \Hom(L_n,R_n\oplus L_n)$ take $1\in k$ to the homomorphism which sends $l\in L_n$ to $(0,l)\in R_n\oplus L_n$. Thus $k^l\xrightarrow{f_0}k^m\rightarrow M_0\rightarrow 0$ is exact and $M_0$ is finite-dimensional.

\end{proof}

In order to prove \ref{descfpind}, we prove the following two conditions by induction. Note that \ref{descfpind} follows from $B_n$ by induction on $n$.

\bigskip

\textbf{$A_n$}: If $M\in\mod-R_n$, $M_0$ a finite-dimensional $k$-vector space and $\Gamma:M_0\rightarrow \Hom(L_n,M)$ is an injective $k$-vector space homomorphism then there is a basis $v_1,\ldots,v_n$ for $M_0$ and orthogonal idempotent endomorphisms $e_i$ of $M$ such that $e_i\Gamma(v_i)=\Gamma(v_i)$ and $e_iM$ is indecomposable.

\bigskip

\textbf{$B_n$}: All finitely presented modules over $R_n$ are direct sums of modules of the form $F_1M, F_0M$ and $T_n$ where $M$ is a finitely presented indecomposable module over $R_{n-1}$ and $\dim\Hom(L_n,N)\leq 1$ for all indecomposable finitely presented modules $N$ over $R_n$.

\begin{lemma}\label{HomF1F0is0}
If $K\in\Mod-R_n$ then
$\Hom(F_1L_n,F_0K)=0$. Hence $F_1F_0K\cong F_0F_0K$.
\end{lemma}
\begin{proof}
Suppose $(f_0,f_1)$ is a homomorphism from $F_1L_n$ to $F_0K$. Then $f_0=0$ and hence $f_1\circ 1_{L_n}=0$. So $f_1=0$.
\end{proof}

%

\begin{remark}\label{twistingrepsbyauto}
Suppose $M\in\mod-R_n$, $M_0$ a finite-dimensional $k$-vector space, $\Gamma:M_0\rightarrow \Hom(L_n,M)$ is an injective $k$-vector space homomorphism and $\alpha$ is an automorphism of $M$. There is a basis $v_1,\ldots,v_n$ for $M_0$ and orthogonal idempotent endomorphisms $e_i$ of $M$ such that $e_i\Gamma(v_i)=\Gamma(v_i)$ and $e_iM$ is indecomposable if and only if there is a basis $v_1,\ldots,v_n$ for $M_0$ and orthogonal idempotent endomorphisms $e_i$ of $M$ such that $e_i\alpha\Gamma(v_i)=\alpha\Gamma(v_i)$ and $e_iM$ is indecomposable.
\end{remark}

\begin{lemma}\label{Andescfp}
If $A_n$ holds then all finitely presented right $R_{n+1}$-modules are direct sums of modules of the form $T_{n+1}:=(k,0,0)$, $(0,M_1,0)$ and $(k,M_1,\Gamma)$ where $M_1$ is a finitely presented indecomposable right $R_n$-module and $\Gamma$ is an injective $k$-vector space homomorphism.
\end{lemma}
\begin{proof}
Let $(M_0,M_1,\Gamma)$ be an $R_{n+1}$-module.
If $(M_0,M_1,\Gamma)$ is finitely presented then $M_1$ is finitely presented and $M_0$ is finite-dimensional by \ref{fpfp}. If $\Gamma$ is not injective then \[(M_0,M_1,\Gamma)\cong (\ker\Gamma,0,0)\oplus (M_0/\ker\Gamma,M_1,\Gamma).\] So, without loss of generality, we may assume $\Gamma$ is injective. By $A_n$ there exists a basis $v_1,\ldots,v_n$ for $M_0$ and orthogonal idempotent endomorphisms $e_1,\ldots,e_n$ of $M$ such that $e_i\Gamma(v_i)=\Gamma(v_i)$. For each $1\leq i\leq n$, let $t_i:M_0\rightarrow M_0$ be a $k$-linear map such that $t_i(v_i)=v_i$ and $t_i(v_j)=0$ if $i\neq j$. So $(t_1,e_1),\ldots,(t_n,e_n)$ are orthogonal idempotents for $(M_0,M_1,\Gamma)$. Thus \[(M_0,M_1,\Gamma)=(0,(1-\sum_{i=1}^ne_i)M_1,0)\oplus\bigoplus_{i=1}^n(v_ik,e_iM_1,\Gamma|_{v_ik})\] as required.
\end{proof}


%
%
%
%
%
%

\begin{lemma}\label{dimHomLnLn}
For all $n\in\N$, $\dim\Hom(L_n,L_n)=1$.
\end{lemma}
\begin{proof}
Since $\dim\Hom(L_0,L_0)=1$ and $F_1$ is full and faithful, $\dim\Hom(L_n,L_n)=1$ for all $n\in\N$.
\end{proof}

\begin{lemma}\label{AnforTn}
Let $M_0$ be a $k$-vector space, $v_1,\ldots,v_m$ a basis for $M_0$ and $\Gamma:M_0\rightarrow \Hom(L_n,T_n^l)$ be an injective $k$-vector space homomorphism. There exist $e_{1},\ldots e_m$ orthogonal idempotent endomorphism of $T_n^l$ such that $e_i\Gamma(v_i)=\Gamma(v_i)$ and $e_iT_n^l$ is indecomposable.

\end{lemma}

\begin{proof}
Let $(\alpha_1,0)=\Gamma(v_1),\ldots,(\alpha_m,0)=\Gamma(v_m)$ where each $\alpha_i$ is a $k$-linear map from $\Hom(L_{n-1},L_{n-1})$ to $k^l$. Since $\Gamma$ is injective and $\dim\Hom(L_{n-1},L_{n-1})=1$, it follows that $l\geq m$ and $\alpha_1(Id_{L_{n-1}}),\ldots,\alpha_m(Id_{L_{n-1}})\in k^l$ are linearly independent. Let $\epsilon_1,\ldots,\epsilon_m$ be the idempotent endomorphisms of $k^l$ such that $\epsilon_i\alpha_i(Id_{L_{n-1}})=\alpha_i(Id_{L_{n-1}})$ and $\dim\im\epsilon_i=1$. Let $e_1=(\epsilon_1,0),\ldots,e_m=(\epsilon_m,0)$. By definition, $e_1,\ldots,e_m$ have the required properties.
\end{proof}

\begin{proposition}
For all $n\geq 1$, $A_{n-1}$ and $B_{n}$ imply $A_n$.
\end{proposition}
\begin{proof}
Let $M_1$ be a finitely presented module over $R_n$, $M_0$ a $k$-vector space and $\Gamma:M_0\rightarrow \Hom(L_n,M_1)$ injective. By $B_n$, $M_1=F_1N\oplus F_0K\oplus T_n^l$ where $K,N\in\mod-R_{n-1}$ and $l\in\N_0$.

By \ref{HomF1F0is0}, $\Hom(F_1L_{n-1},F_0K)=0$, so we may as well assume $K=0$.

Let $v_1,\ldots,v_n$ be a basis for $M_0$ such that \[\Gamma(v_1)=(f_1,w_1),\ldots,\Gamma(v_m)=(f_m,w_m)\] and \[\Gamma(v_{m+1})=(0,w_{m+1}),\ldots,\Gamma(v_n)=(0,w_n)\] where $f_1,\ldots,f_m\in \Hom(L_n,F_1N)$ are linearly independent over $k$ and $w_1,\ldots,w_n\in\Hom(L_n,T_n^l)$.

Since $F_1$ is full, there exist $f_1^*,\ldots,f_m^*\in\Hom(L_{n-1},N)$ such that $F_1f_1^*=f_1,\ldots,F_1f_m^*=f_m$.

Let $w_1^*,\ldots,w_n^*$ be such that $w_i=(t_i,0)$ and $w_i^*=t_i(Id_{L_{n-1}})$. Note that since $\Hom(L_{n-1},L_{n-1})$ is $1$-dimensional, $t_i$ and hence $w_i$ is determined by $w_i^*=t_i(Id_{L_{n-1}})$.

Here is a diagram for $\Gamma(v_i)$:

\[\xymatrix@R=60pt@C=80pt{
  \Hom(L_{n-1},L_{n-1}) \ar[d]_{\substack{g\\\downmapsto\\(f^*_i\circ g,t_i(g))}} \ar[r]^{Id} & \Hom(L_{n-1},L_{n-1}) \ar[d]^{\substack{g\\\downmapsto\\f^*_i\circ g}} \\
  \Hom(L_{n-1},N)\oplus k^l \ar[r]^{(\text{ Id },0)} & \Hom(L_{n-1},N)   }\]

Let $\alpha=(\alpha_0,\alpha_1):F_1N\oplus T_n^l\rightarrow F_1N\oplus T_n^l$ be such that $\alpha_0\circ(f_i^*,0)=(f_i^*,-w_i^*)$ for $i=1,\ldots, m$, $\alpha_0\circ (0,w_i^*)=(0,w_i^*)$ for $i=m+1\ldots,n$ and $\alpha_1=\text{Id}_N$.

So $\alpha\circ \Gamma(v_1)=(f_1,0),\ldots \alpha\circ\Gamma(v_m)=(f_m,0)$ and $\alpha\circ\Gamma(v_{m+1})=(0,w_{m+1}),\ldots,\alpha\circ\Gamma(v_{n})=(0,w_{n})$. By \ref{twistingrepsbyauto}, we may replace $\Gamma$ by $\alpha\circ\Gamma$.

Let $M_0'$ be the span of $v_1,\ldots,v_m$. By definition of $M_0'$, if $u\in M_0'$ then $\Gamma(u)=(\Gamma_0(u),0)\in \Hom(L_n,F_1N)\oplus\Hom(L_n,T_n^l)= \Hom(L_n,F_1N\oplus T_n^l)$. Let $\Delta:M_0'\rightarrow \Hom(L_{n-1},N)$ be defined by setting $F_1\Delta(u)=\Gamma_0(u)$. By $A_{n-1}$, there exists $e_1,\ldots,e_m$ orthogonal idempotent endomorphisms of $N$ such that $e_i\Delta(v_i)=\Delta(v_i)$ and $e_iN$ is indecomposable for $1\leq i\leq m$. Thus $F_1(e_i)\Gamma(v_i)=\Gamma(v_i)$ and $F_1(e_i)F_1(N)=F_1(e_iN)$ which is indecomposable.

Let $M_0''$ be the span of $v_{m+1},\ldots,v_n$. By \ref{AnforTn}, there exist $e_{m+1},\ldots e_n$ orthogonal idempotent endomorphism of $T_n^l$ such that $e_i\Gamma(v_i)=\Gamma(v_i)$ and $e_iT_n^l$ is indecomposable.

So $\sigma_1=(e_1,0),\ldots,\sigma_m=(e_m,0)$ and $\sigma_{m+1}=(0,e_{m+1}),\ldots,\sigma_n=(0,e_n)$ are orthogonal idempotent endomorphisms of $F_1N\oplus T_n^l$ such that $\sigma_i\Gamma(v_i)=\Gamma(v_i)$ and $\sigma_i(F_1N\oplus T_n^l)$ is indecomposable.
\end{proof}

\begin{proposition}\label{AnandBnimplyBnplus1}
For all $n\in\N$, $B_n$ and $A_n$ imply $B_{n+1}$.
\end{proposition}
\begin{proof}
By \ref{Andescfp}, $A_n$ implies that each finitely presented module over $R_{n+1}$ is a direct sum of modules of the form $T_{n+1}, (k,M_1,\Gamma)$ and $F_1M_1=(0,M_1,0)$ where $\Gamma$ is injective and $M_1$ is an indecomposable $R_n$-module. So in order to show that the first clause of $B_{n+1}$ is true, we need now consider modules of the form $(k,M_1,\Gamma)$.  By $B_n$, $M_1$ is either $F_1N$, $F_0K$ or $T_n$ where $N,K$ are indecomposable $R_{n-1}$-modules.

Since $\Hom(L_n,F_0K)=0$, if $M_1=F_0K$ then $\Gamma$ is not injective.

If $M_1=T_n$ then $(k,T_n,\Gamma)$ is isomorphic to $F_1T_n$ since $\dim_k\Hom(L_n,T_n)=1$ i.e. $(\Gamma,\text{Id}_{T_n}):(k,T_n,\Gamma)\rightarrow F_1T_n$ is an isomorphism.

Now suppose that $M_1=F_1N$ for $N$ an indecomposable $R_{n-1}$-module. Since $\Gamma:k\rightarrow \Hom(L_n,F_1N)$ is injective, $\dim\Hom(L_n,F_1N)\neq 0$. By $B_{n-1}$, $\dim\Hom(L_{n-1},N)\leq 1$. So $\dim\Hom(L_n,F_1N)=1$. Thus $(\Gamma,\text{Id}_{N}):(k,N,\Gamma)\rightarrow F_1N$ is an isomorphism.

It now remains to show that $\dim\Hom(L_{n+1},M)\leq 1$ for all indecomposable $M\in \mod-R_{n+1}$. If $M=F_0K$ then $\dim\Hom(L_{n+1},M)=0$. If $M=F_1N$ then $\dim\Hom(L_{n+1},M)=\dim\Hom(L_{n},N)\leq 1$ by $B_n$. If $M=T_{n+1}$ then $\dim\Hom(L_{n+1},M)\leq 1$ follows from \ref{dimHomLnLn} and the definition of $T_{n+1}$.

\end{proof}

We now consider the base cases, $A_0$ and $B_1$.

\begin{lemma}\label{basecaseA}
Let $V$ be a discrete valuation domain with maximal ideal $\mfrak{m}$, $M$ a finitely presented $V$-module and $M_0\subseteq\Hom(V/\mfrak{m},M)$ a $V/\mfrak{m}$-vector subspace. There exists  $v_1,\ldots,v_n$ a basis for $M_0$ as a $V/\mfrak{m}$-vector space and orthogonal idempotent endomorphisms of $M$ such that $e_iM$ is indecomposable and $e_iv_i=v_i$ for $1\leq i\leq n$.
\end{lemma}
\begin{proof}

Since $V$ is a principal ideal domain, $M=M'\oplus V^m$ where $M'$ is a torsion module. Since $\Hom(V/\mfrak{m},V)=0$, we may replace $M$ by $M'$. Let $W= \{f(a)\in M \st a\in V/\mfrak{m} \text{ and } f\in M_0\}$. Note that $W$ is a submodule of $M$. Let $p\in V$ generate $\mfrak{m}$ and let $L=\{m\in M \st 0\neq mp^l\in W \text{ for some }l\in\N\}$.

Since $V$ is an RD-ring (see \cite[section 2.4.2]{PSL}), $L$ is pure in $M$. Since $M$ and hence $L$ is finite-length, $L$ is pure-injective. Therefore $L$ is a direct summand of $M$.

The lemma now follows from the structure theorem for finitely generated module over principle ideal domains.

%
%
%

\end{proof}

\begin{lemma}\label{basecaseB}
$B_1$ holds.
\end{lemma}
\begin{proof}
We need to show that all finitely presented modules over $R_1$ are direct sums of modules of the form $F_0M:=(0,M,0)$, $F_1M$ and $T_1:=(k,0,0)$ where $M$ is a finitely presented indecomposable module over $V$. Let $(M_0,M,\Gamma)$ be an arbitrary finitely presented module over $R_1$. As usual, we may assume $\Gamma$ is injective. Letting $e_1,...,e_n$ be as in \ref{basecaseA}, we have that
\[(M_0,M,\Gamma)\cong (0,(1-\sum_{i=1}^ne_i)M,0)\oplus \bigoplus_{i=1}^n(v_ik,e_iM,\Gamma|_{v_ik}). \] The module $(0,(1-\sum_{i=1}^ne_i)M,0)$ is equal to $F_0(1-\sum_{i=1}^ne_i)M$ and for each $1\leq i\leq n$, $(v_ik,e_iM,\Gamma|_{v_ik})$ is isomorphic to $F_1e_iM$.

That $\Hom(L_1, N)$ is $1$-dimensional is proved exactly as in the proof of \ref{AnandBnimplyBnplus1}.
\end{proof}

\begin{cor}
For all $n\geq 0$ and $m\geq 1$, $A_n$ and $B_m$ hold.
\end{cor}

\begin{cor}
All indecomposable finitely presented modules over $R_n$ have local endomorphism rings.
\end{cor}

\section{Indecomposable pure-injectives and the Ziegler spectrum}\label{Indpis}
Throughout this section, let $V$ be a valuation domain with maximal idea $\mfrak{m}$.

\begin{lemma}
All finitely presented indecomposable right $R_{n+1}$-modules are pure-injective except for $F_0^{n+1}V$. The module $F_0^{n+1}V$ is pure-injective if and only if $V$ is pure-injective as a right module over itself.
\end{lemma}
\begin{proof}
This follows directly from \ref{descfpind}, the fact that the functors $F_0$ and $F_1$ preserve pure-injectivity and that $(k,0,0)$ is finite-length and hence pure-injective.
\end{proof}

\begin{proposition}\label{piinimF0F1plusT}
Every indecomposable pure-injective module over $R_{n+1}$ is of the form $F_0N$, $F_1N$ or $T_n:=(k,0,0)$ for some indecomposable pure-injective $R_n$-module $N$.
\end{proposition}
\begin{proof}
Since all finitely presented indecomposable modules over $R_{n+1}$ are pure-injective except $F_0^{n+1}V$, the set of finitely presented pure-injective modules together with $F_0^{n+1}\widehat{V}$ is dense in $\Zg_{R_{n+1}}$. Since $F_0$ and $F_1$ commute with direct limits and products, the images of $F_0$ and $F_1$ are definable subcategories of $\Zg_{R_{n+1}}$ after closing under direct summands (see \cite[3.8]{Abcatsanddefaddcats}). Thus $\text{Add}(\im F_0)\cap\Zg_{R_{n+1}}=:\mcal{C}_0$ and $\text{Add}(\im F_1)\cap\Zg_{R_{n+1}}=:\mcal{C}_1$ are closed subsets of $\Zg_{R_{n+1}}$. The point $(k,0,0)$ is a closed point of $\Zg_{R_{n+1}}$. Since $\mcal{C}_1\cup\mcal{C}_2\cup\{(k,0,0)\}$ is a dense closed subset of $\Zg_{R_{n+1}}$ it is all of $\Zg_{R_{n+1}}$. Thus all indecomposable pure-injecitve $R_{n+1}$-modules are of the required form.
\end{proof}

\begin{theorem}\label{indoveronepointextensions} The indecomposable pure-injective modules over $R_n$ are
\begin{enumerate}[(i)]
\item $F_0^pF_1^lN$ where $N=V/\mfrak{m}^m$ or $N=E(V/\mfrak{m})$ and $p+l=n$
\item $F_0^nN$ where $N=\widehat{V}$ or $N=Q(V)$
\item $F_0^pF_1^lT_m$ where $p+l+m=n$
\end{enumerate}
There is no redundancy in this list.

Further, the modules in (i), (iii) and $F_0^nQ(V)$ are $\Sigma$-pure-injective, $F_0^nQ(V)$ is finite-length over its endomorphism ring and $F_0^n\widehat{V}$ has acc on pp-definable subsets.
\end{theorem}
\begin{proof}
By \ref{piinimF0F1plusT} and the fact that $F_1F_0\cong F_0^2$, the indecomposable pure-injectives over $R_n$ are
\begin{enumerate}[(i)]
\item $F_0^pF_1^lN$ where $N=V/\mfrak{m}^m$, $N=E(V/\mfrak{m})$, $N=\widehat{V}$ or $N=Q(V)$ and $p+l=n$
\item $F_0^pF_1^lT_m$ where $p+l+m=n$.
\end{enumerate}

Since $\Hom(V/\mfrak{m},\widehat{V})=0$ and $\Hom(V/\mfrak{m},Q(V))=0$, $F_1\widehat{V}\cong F_0\widehat{V}$ and $F_1Q(V)\cong F_0Q(V)$. So by \ref{HomF1F0is0}, $F_0^pF_1^l\widehat{V}\cong F_0^n\widehat{V}$ and  $F_0^pF_1^lQ(V)\cong F_0^nQ(V)$ when $p+l=n$.

It remains to show that there is no redundancy in the list of pure-injectives given in the statement of the theorem. By applying the functor $r$ $n$-times, we see that if $N,M$ are non-isomorphic then $F_0^pF_1^{n-p}N$ is not isomorphic to $F_0^lF_1^{n-l}M$. If we apply the functor $r$ $n$-times to a module in $(iii)$ then we get the zero module. Thus the $3$ points of the list are pairwise disjoint and $F_0^n\widehat{V}$ is not isomorphic to $F_0^nQ(V)$.

Now suppose that $F_0^pF_1^{n-p}N\cong F_0^lF_1^{n-l}N$ and $p<l$. Then $F_1^{n-p}N\cong F_0^{l-p}F_1^{n-l}N$. In order to show that these two modules are not isomorphic, it is enough to show that for $N$ from $(i)$ and $m\geq 1$, $\Hom(L_{m-1},F_1^{m-1}N)\neq 0$. This is true since $F_1$ is full and faithful and $\Hom(L_0,N)\neq 0$.

We leave showing that there is no redundancy in $(iii)$ to the reader.

\end{proof}

\begin{remark}
If $N$ is an indecomposable pure-injective over $R_n$ then its lattice of pp-definable subgroups has m-dimension 1.
\end{remark}

\begin{remark}
The functors $F_0$ and $F_1$ induce homeomorphisms from $\Zg_{R_n}$ to $\Zg_{R_{n+1}}\cap\im F_0$ and $\Zg_{R_{n+1}}\cap\im F_1$ respectively. Since both functors are full and faithful, this follows from \cite[6.1]{Krausegenmods}.
\end{remark}

\begin{remark}
A set $X$ is closed if $X\cap\im F_0$ and $X\cap\im F_1$ is closed. Thus we can understand the Ziegler topology inductively. In particular $F_0^nQ$ is a closed point and $\{N,F_0^nQ\}$ is the closure of $N$ for any non-finite-length indecomposable pure-injective $N$.
\end{remark}

\begin{remark}
The m-dimension of $\pp_{R_n}^1$ is $2$. The functor $F_1\oplus F_0:\Mod-R_{n}\rightarrow \Mod-R_{n+1}$ is such that $\langle (F_1\oplus F_0)\Mod-R_{n}, (k,0,0)\rangle=\Mod-R_{n+1}$. Any pp-pair which is finite-length with respect to $\langle F_1\oplus F_0\Mod-R_{n}\rangle$ is finite-length with respect to $\langle (F_1\oplus F_0)\Mod-R_{n}, (k,0,0)\rangle$. Thus, the m-dimension of $\pp^1_{R_n}$ is equal to the m-dimension of $\pp^1_{R_{n+1}}$ for all $n\in\N_0$.
\end{remark}

\section{Ziegler closures of generalised ray tubes}

Throughout this section we will take $R_0:=k[[x]]$ and $\mcal{A}$ a finite-dimensional $k$-algebra.

Following \cite{Krausegenmods}, a generalised (stable) tube is a sequence of tuples $\mcal{T}:=(M_i,\phi_i,\psi_i)_{i\in\N_0}$ where each $M_i$ is an $\mcal{A}$-module and $\phi_i:M_{i+1}\rightarrow M_i$ and $\psi_i:M_i\rightarrow M_{i+1}$ are $\mcal{A}$-homomorphisms such that $M_0=0$ and for every $i\in\N$,
\[\xymatrix@=20pt{
  M_i \ar[d]_{\psi_i} \ar[r]^{\phi_{i-1}} & M_{i-1} \ar[d]^{\psi_{i-1}} \\
  M_{i+1} \ar[r]^{\phi_i} & M_i   }\] is a pull back and a push out.

A generalised ray tube $((M_i,\phi_i,\psi_i)_{i\in\N_0},(P^i,\alpha_i)_{i=1}^n)$ is a generalised tube $\mcal{T}=(M_i,\phi_i,\psi_i)_{i\in\N_0}$ together with a finite sequence of $\mcal{A}$-modules $P^1,P^2,\ldots,P^n$ and embeddings $\alpha^1:M_1\rightarrow P^1$, $\alpha^2:P^1\rightarrow P^2$,\ldots,$\alpha^n:P^{n-1}\rightarrow P^n$.

To each generalised ray tube we attach a set
\[\{P^i_j\st 1\leq i\leq n \text{ and } j\geq 1\}\cup\{M_i \st i\geq 1\}\]
 of finitely presented modules, by induction, which we will say are \textbf{in the tube}.

For each $1\leq i\leq n$, let $P_1^i=P^i$ and $\alpha_1^i=\alpha^i$. For each $j\geq 1$, let $P^1_j$, $\alpha^1_j$ and $\overline{\psi^1_j}$ be such that
\[\xymatrix@=20pt{
  M_j \ar@{^{(}->}[d]_{\alpha^1_{j}} \ar@{^{(}->}[r]^{\psi_j} & M_{j+1} \ar[d]^{\alpha^1_{j+1}}  \\
  P_j^1 \ar[r]_{\overline{\psi_j^1}} &  P_{j+1}^1             }\] is a pushout.

Since $\alpha^1$ and $\psi^1$ are embeddings, the pushout with $j=1$ is also a pullback. Moreover, $\alpha_2^1$ and $\overline{\psi_1^1}$ are embeddings. By induction, it follows that for all $j\geq 1$, the above pushout is also a pullback and that $\alpha_{j+1}^1$ and $\overline{\psi_j^1}$ are embeddings.

For each $1<i\leq n$ and $1\leq j$ let $P^i_j$, $\alpha^i_j$ and $\overline{\psi_i^j}$ be such that
\[\xymatrix@=20pt{
  P_j^{i-1} \ar@{^{(}->}[d]_{\alpha^i_{j}} \ar@{^{(}->}[r]^{\overline{\psi_j^{i-1}}} & P_{j+1}^{i-1} \ar[d]^{\alpha^i_{j+1}}  \\
  P_j^i \ar[r]_{\overline{\psi_j^i}} &  P_{j+1}^i             }\] is a pushout.

Note that the above pushouts are also pullbacks, and the morphisms $\alpha_j^i$ and $\overline{\psi^i_j}$ are embeddings.

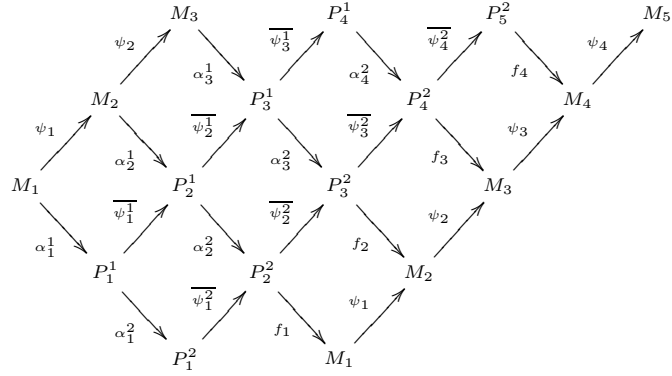
\tiny
\[\xymatrix@R=18pt@C=12pt{
& & M_3 \ar[dr]_{\alpha_3^1} & & P_4^1 \ar[dr]_{\alpha^2_4} & & P^2_5 \ar[dr]_{f_4}& & M_5 \\
& M_2 \ar[ur]^{\psi_2}\ar[dr]_{\alpha_2^1} & & P_3^1\ar[ur]^{\overline{\psi_3^1}}\ar[dr]_{\alpha_3^2} & & P^2_4 \ar[ur]^{\overline{\psi_4^2}}\ar[dr]_{f_3}& & M_4\ar[ur]^{\psi_4}&\\
M_1 \ar[ur]^{\psi_1}\ar[dr]_{\alpha_1^1}& & P_2^1\ar[ur]^{\overline{\psi_2^1}}\ar[dr]_{\alpha^2_2} & & P^2_3\ar[ur]^{\overline{\psi^2_3}}\ar[dr]_{f_2}& & M_3\ar[ur]^{\psi_3}&&\\
& P_1^1 \ar[ur]^{\overline{\psi_1^1}}\ar[dr]_{\alpha_1^2} & & P_2^2\ar[ur]^{\overline{\psi_2^2}}\ar[dr]_{f_1}& & M_2\ar[ur]^{\psi_2} & &&\\
& & P_1^2 \ar[ur]^{\overline{\psi_1^2}} & & M_1\ar[ur]^{\psi_1} & & & &&
}\]
\normalsize

From the below lemma it follows that the cokernels of $\psi_1$, $\overline{\psi_1^1}$,\ldots ,$\overline{\psi_1^n}$ are isomorphic and thus all isomorphic to $M_1$.  As in the diagram above we can complete the picture by letting $f_1$ be a cokernel of $\overline{\psi_1^n}$. It follows that $f_1\circ\alpha_2^n\circ\cdots\circ\alpha^2_1$ is the cokernel of $\psi_1$. Since $\phi_2$ is also a cokernel of $\psi_1$, by postcomposing $f_1$ with an isomorphism, we may assume that $f_1\circ\alpha_2^n\circ\cdots\circ\alpha_2^1=\phi_2$.

\begin{lemma}
Suppose that the following diagram occurs in an abelian category and that the left hand square in the following diagram is both a pushout and a pullback. The induced morphism $\epsilon:\coker \ \alpha \rightarrow\coker \ \beta$ is an isomorphism.

\xymatrix@=20pt{
  A \ar[d]^{\gamma} \ar[r]^{\alpha} & B \ar[d]^{\delta} \ar[r] & \coker \ \alpha \ar[d]^{\epsilon} \ar[r] & 0  \\
  D \ar[r]^{\beta} & E \ar[r] & \coker \ \beta \ar[r] & 0   }

\end{lemma}
\begin{proof}
Since the left hand square is a pullback, \cite[10.1.2]{Handbookcatalg2} implies that $\epsilon$ is a monomorphism. That $\epsilon$ is surjective follows from the fact that $E$ is a pushout and thus $\im \beta+\im \delta=E$.
\end{proof}

Consider the following pushout.

\[\xymatrix@=20pt{
  P^n_2 \ar[d]_{f_1} \ar[r]^{\overline{\psi_2^n}} & P^n_3 \ar[d]^{f_2} \\
  M_1 \ar[r]^{g_1} & N   }\]

By putting pushouts side by side, we get that

\[\xymatrix@=20pt{
  M_2 \ar[d]_{\phi_2} \ar[r]^{\psi_2} & M_3 \ar[d]^{f_2\circ\alpha_3^n\circ\ldots\circ\alpha_3^1} \\
  M_1 \ar[r]^{g_1} & N   }\]

is also a pushout.

Thus $N$ is isomorphic to $M_2$ and by postcomposing by an isomorphism, we may assume $g_1=\psi_1$ and $f_2\circ\alpha_3^n\circ\ldots\circ\alpha_3^1=\phi_3$.
Continuing in this way we can complete the whole picture as in the diagram.

\bigskip

As in \cite{Krausegenmods}, let $\widehat{M}$ be the inverse limit of
\[M_1\xleftarrow{\phi_1} M_2\xleftarrow{\phi_2} M_3\xleftarrow{\phi_3}\ldots\] and let $u_1:\widehat{M}\rightarrow M_1$ be the induced  morphism to $M_1$.

Our main task in this section is to equip $\hat{M}\oplus\bigoplus_{i=1}^nP_i$ with the structure of a projective left $R_n$-module.

%
%
%

For each $n\in\N_0$, define $K_n\subseteq \Hom_{\mcal{A}}(\widehat{M}\oplus\bigoplus_{i=1}^nP^i,P^{n+1})$ inductively by setting $K_0=k\beta_1$ where $\beta_1=\alpha_1u_1$ and $K_{n+1}$ to be the subset of elements of $\Hom_{\mcal{A}}(\widehat{M}\oplus\bigoplus_{i=1}^{n+1}P^i,P^{n+2})$ of the form
\[(\alpha_{n+2}f,\lambda\alpha_{n+2}): \left(
                                                    \begin{array}{c}
                                                      m \\
                                                      p_{n+1} \\
                                                    \end{array}
                                                  \right)\mapsto \alpha_{n+2}\circ f(m)+\lambda\alpha_{n+2}(p_{n+1})\] where $f\in K_n \text{ and }\lambda\in k$.

For each $n\in \N_0$, we define $i_n:R_n\hookrightarrow \End(\widehat{M}\oplus\bigoplus_{i=1}^nP_i)$ and $\Delta_n:L_n\rightarrow K_n$ an isomorphism of $R_n$-modules when $K_n$ is viewed as a right $R_n$-module via $i_n$. Let $i_0:R_0\hookrightarrow \End(\widehat{M})$ be as in \cite{Krausegenmods} and $\Delta_0:L_0\rightarrow K_0$ be defined by $\Delta_0:\lambda\in k\mapsto \lambda\beta_1$. If $r\in R_n,l\in L_n$ and $\lambda\in k$ then let $i_{n+1}$ send $\left(
                                                                                                                               \begin{array}{cc}
                                                                                                                                 r & 0 \\
                                                                                                                                 l & \lambda \\
                                                                                                                               \end{array}
                                                                                                                             \right)
$ to the $\widehat{M}\oplus\bigoplus_{i=1}^{n+1},P_i$ endomorphism
\[\left(
    \begin{array}{c}
      m \\
      p_{n+1} \\
    \end{array}
  \right)\mapsto \left(
                   \begin{array}{c}
                     i_n(r)(m) \\
                     \Delta_n(l)(m)+\lambda p_{n+1} \\
                   \end{array}
                 \right).
\]
Let $\Delta_{n+1}:L_{n+1}\rightarrow K_{n+1}$ be defined by $(\lambda 1_{L_{n}},l)\mapsto (\alpha_{n+1}\circ\Delta_n(l),\lambda\alpha_{n+1})$.

\begin{proposition}
For each $n\in\N_0$, $i_n:R_n\hookrightarrow \End(\widehat{M}\oplus\bigoplus_{i=1}^nP_i)$ is an embedding of rings and $\Delta_n$ is an isomorphism of $R_n$-modules.
\end{proposition}
\begin{proof}
That $i_0$ is an embedding is already covered in \cite{Krausegenmods} and $\Delta_0$ is defined to be an isomorphism.

Suppose that $i_{n}:R_n\rightarrow \End(\widehat{M}\oplus\bigoplus_{i=1}^nP_i)$ is an embedding of rings and $\Delta_n$ is an isomorphism of $R_n$-modules.

A quick computation shows that $i_{n+1}$ is an embedding of rings.

If $\Delta_{n+1}(v,l)=0$ then $\alpha_{n+2}\circ\Delta_n(l)=0$ and $v=0$. Since $\alpha_{n+2}$ is injective, $\Delta_{n+1}(v,l)=0$ implies $\Delta_n(l)=0$. So $\Delta_{n+1}$ is injective. Since $K_{n+1}$ and $L_{n+1}$ are both finite-dimensional of the same dimension over $k$, $\Delta_{n+1}$ is an isomorphism.

\end{proof}

Now that we have equipped $\widehat{M}\oplus\bigoplus_{i=1}^nP_i$ with the structure of a left $R_n$-module, we need to show that it is a projective left $R_n$-module.

We inductively define a a set of orthogonal idempotents for $R_n$.
Let $c_1=\left(
         \begin{array}{cc}
           1 & 0 \\
           0 & 0 \\
         \end{array}
       \right)\in R_1$ and $e_1^1=\left(
                                    \begin{array}{cc}
                                      0 & 0 \\
                                      0 & 1 \\
                                    \end{array}
                                  \right)\in R_1$. Let $c_{n+1}=\left(
                                                              \begin{array}{cc}
                                                                c_n & 0 \\
                                                                0 & 0 \\
                                                              \end{array}
                                                            \right)
                                  $, $e_i^{n+1}=\left(
                                                 \begin{array}{cc}
                                                   e_i^n & 0 \\
                                                   0 & 0 \\
                                                 \end{array}
                                               \right)
                                  $ for $1\leq i\leq n$ and $e_{n+1}^{n+1}=\left(
                                                                             \begin{array}{cc}
                                                                               0 & 0 \\
                                                                               0 & 1 \\
                                                                             \end{array}
                                                                           \right)$.
Note that $c_n+\sum_{i=1}^ne_i^n=1$, $c_ne_i^n=e_i^nc_n=0$, $e_i^ne_j^n=e_j^ne_i^n=0$, $c_nc_n=c_n$ and $e_i^ne_i^n=e_i^n$.

So for each $n$, $e_1^n,\ldots, e_n^n,c_n$ are a set of orthogonal idempotents.

\begin{proposition}\label{bimoduleisprojective}
As a left $R_n$-module, $\widehat{M}\oplus\bigoplus_{i=1}^nP_i$ is isomorphic to \[(R_nc_n)^{d_0}\oplus (R_ne_1^n)^{d_1}\oplus\ldots\oplus(R_ne_n^n)^{d_n}\] where $d_0=\dim M_1, d_1=\dim P_1-\dim M_1$ and for $i>1$, $d_i=\dim P_i-\dim P_{i-1}$.
\end{proposition}

Before we prove this proposition, we describe how to understand left modules over (right) one-point extensions. Let $\mcal{G}$ be the category with objects $(M,V,\Upsilon)$ where $M$ is a left $R$-module, $V$ is a $k$-vector space and $\Upsilon: {_k}L\otimes M\rightarrow V$ is a $k$-linear map and morphisms $(f,g):(M_1,V_1,\Upsilon_1)\rightarrow (M_2,V_2,\Upsilon_2)$ where $f:M_1\rightarrow M_2$ is a left $R$-module map and $g:V_1\rightarrow V_2$ is $k$-linear and
\[\xymatrix@=20pt{
  L\otimes M_1 \ar[d]_{Id_L\otimes f} \ar[r]^{\Upsilon_1} & V_1 \ar[d]^{g} \\
  L\otimes M_2 \ar[r]^{\Upsilon_2} & V_2   }\] commutes. The category $\mcal{G}$ is equivalent to the category of left modules over $\left(
                                                                                                                                      \begin{array}{cc}
                                                                                                                                        R & 0 \\
                                                                                                                                        L & k \\
                                                                                                                                      \end{array}
                                                                                                                                    \right)$. This equivalence is given by the functors:
\[\left(
    \begin{array}{cc}
      R & 0 \\
      L & k \\
    \end{array}
  \right)-\Mod\rightarrow \mcal{G}: M\mapsto (\left(
                                                \begin{array}{cc}
                                                  1 & 0 \\
                                                  0 & 0 \\
                                                \end{array}
                                              \right)M, \left(
                                                          \begin{array}{cc}
                                                            0 & 0 \\
                                                            0 & 1 \\
                                                          \end{array}
                                                        \right)M, l\otimes m\mapsto \left(
                                                                                     \begin{array}{cc}
                                                                                       0 & 0 \\
                                                                                       l & 0 \\
                                                                                     \end{array}
                                                                                   \right)m
                                                          )
\] and
\[\mcal{G}\rightarrow \left(
    \begin{array}{cc}
      R & 0 \\
      L & k \\
    \end{array}
  \right)-\Mod: (M,V,\Upsilon)\mapsto M\oplus V \] with

 \[\left(
     \begin{array}{cc}
       r & 0 \\
       l & a \\
     \end{array}
   \right)
 \left(
   \begin{array}{c}
     m \\
     v \\
   \end{array}
 \right)
 =\left(
    \begin{array}{c}
      rm \\
      av+\Upsilon(l\otimes m) \\
    \end{array}
  \right)
 \] for $r\in R$, $l\in L$ and $a\in k$.

There are $n+1$ indecomposable projective over $R_n$ each corresponding to one of the idempotents $c_n,e^n_1,\ldots,e^n_n$.

As objects in $\mcal{G}$ they are described inductively: $R_{n+1}c_{n+1}$ corresponds to $(R_nc_n,L_nc_n,l\otimes rc_n\mapsto lrc_n)$, $R_{n+1}e_i^{n+1}$ corresponds to $(R_ne^n_i,L_ne_i^n, l\otimes re^n_i\mapsto lre^n_i)$ and $R_{n+1}e_{n+1}^{n+1}$ corresponds to $(0,k,0)$.

\begin{proof}[proof of proposition \ref{bimoduleisprojective}]
As always we prove the statement by induction on $n$. The base case just says that $\widehat{M}$ is isomorphic to $V^{d_0}$ where $d_0=\dim M_1$. This has already be proved in \cite[lemma 8.8]{Krausegenmods}. Suppose the statement is true for $n$.

 So there is an isomorphism $T_n:(R_nc_n)^{d_0}\oplus(R_ne_1^n)^{d_1}\oplus\ldots\oplus(R_ne_n^n)^{d_n}\rightarrow\widehat{M}\oplus\bigoplus_{i=1}^nP_i$ of left $R_n$-modules. Write $((c_n)^{d_0},(e_1^n)^{d_1},\ldots,(e_n^n)^{d_n})$ for the $\sum_{i=0}^nd_i$-tuple with first $d_0$ entries $c_n$, the next $d_1$ entries $e_1^n$ and so on, with final $d_n$ entries $e_n^n$.

Define $\Theta_n:[(L_nc_n)^{d_0}\oplus\ldots\oplus(L_ne_n^n)^{d_n}]\oplus k^{d_{n+1}}\rightarrow P_{n+1}$ to be the map defined by \[(\overline{l},\overline{\lambda}) \mapsto \Delta_n(\overline{l})(T_n(c_n)^{d_0},T_n(e_1^n)^{d_1},\ldots,T_n(e_n^n)^{d_n})\] for all $\overline{l}\in (L_nc_n)^{d_0}\oplus\ldots\oplus(L_ne_n^n)^{d_n}$ and $\overline{\lambda}\in k^{d_{n+1}}$.

The diagram

\tiny
\[\xymatrix@R=30pt@C=50pt{
  L_n\otimes[(R_nc_n)^{d_0}\oplus(R_ne_1^n)^{d_1}\oplus\ldots\oplus(R_ne_n^n)^{d_n}] \ar[d]_{\Delta_n\otimes T_n} \ar[r]^{\Upsilon_n} & [(L_nc_n)^{d_0}\oplus\ldots\oplus(L_ne_n^n)^{d_n}]\oplus k^{d_{n+1}} \ar[d]^{\Theta_n} \\
  K_n\otimes[\widehat{M}\oplus\bigoplus_{i=1}^nP_i] \ar[r]^{\Omega_n} & P_{n+1}   }\]
\normalsize
 commutes since for each idempotent $e\in\{c_n,e_1^n,\ldots,e_n^n\}$ and $r\in R_n$, \[\Omega_n(\Delta_n(l)\otimes T_n( re))=\Delta_n(l)[reT_n(e)]=\Delta_n(lre)[T_n(e)]\] and
 \[\Theta_n(\Upsilon_n(l\otimes re))=\Theta_n(lre)=\Delta_n(lre)[T_n(e)].\]

 The map given by the pair $(\Delta_n\otimes T_n, \Theta_n)$ is not an isomorphism. By induction, one can show that the image of the map which sends $\gamma\otimes m\in K_n\otimes (\widehat{M}\oplus\bigoplus_{i=1}^n)$ to $\gamma(m)\in P_{n+1}$ has dimension $\dim P_{n+1}-d_{n+1}$. The dimension of $(L_nc_n)^{d_0}\oplus(L_ne^n_{1})^{d_1}\oplus\ldots\oplus (L_ne_n^n)^{d_n}$ is $\sum_{i=0}^nd_i=\dim P_{n+1}-d_{n+1}$. Thus, since $\Delta_n\otimes T_n$ is an isomorphism, the kernel of $\Theta_n$ has the same dimension as the cokernel of $\Theta_n$. Thus we can extend $\Theta_n$ to an isomorphism so that the diagram above still commutes. Thus  $\widehat{M}\oplus\bigoplus_{i=1}^{n+1}P_i$ is isomorphic to  $(R_{n+1}c_{n+1})^{d_0}\oplus (R_{n+1}e_1^{n+1})^{d_1}\oplus\ldots\oplus(R_{n+1}e_{n+1}^{n+1})^{d_{n+1}}$.

%
%
%
\end{proof}

\begin{theorem}\label{intfunmappingontogenraytube}
Let $\mcal{A}$ be a finite-dimensional algebra over a field $k$ and $(\mcal{T}, (P^i,\alpha_i)_{i=1}^n)$ be a generalised ray tube in $\mod-\mcal{A}$. The functor \[-\otimes \widehat{M}\oplus\bigoplus_{i=1}^nP^i:\Mod-R_n\rightarrow\Mod-\mcal{A}\] is such that all indecomposable pure-injectives in the closure of the generalised ray tube are direct summands of modules in the image of $-\otimes \widehat{M}\oplus\bigoplus_{i=1}^nP^i$.
\end{theorem}
\begin{proof}
Since $\widehat{M}\oplus\bigoplus_{i=1}^nP^i$ is projective and finitely presented as a left $k[[x]]$-module, it is enough to show, as in \ref{piinimF0F1plusT}, that all finite-dimensional modules in the generalised ray tube, are direct summands of modules in the image of $-\otimes \widehat{M}\oplus\bigoplus_{i=1}^nP^i$. Since $-\otimes \widehat{M}\oplus\bigoplus_{i=1}^nP^i$ is exact, it commutes with pushouts and pullbacks. Thus it is enough to note that $P^1,\ldots,P^n$ are direct summands of $R_n\otimes\widehat{M}\oplus\bigoplus_{i=1}^nP^i=\widehat{M}\oplus\bigoplus_{i=1}^nP^i$ and $F_0^n(k[[x]]/\langle x^i\rangle)\otimes\widehat{M}\oplus\bigoplus_{i=1}^nP^i=M_i$.
\end{proof}
%

Let $((M_i,\phi_i,\psi_i)_{i\in\N_0},(P^i,\alpha_i)_{i=1}^n)$ be a generalised ray tube. Let $M[\infty]:=F_0^nE(V/\mfrak{m})\otimes \widehat{M}\oplus\bigoplus_{i=1}^nP_i$ and $G:=F_0^nQ(V)\otimes \widehat{M}\oplus\bigoplus_{i=1}^nP_i$. For $1\leq i\leq n$, let $P^i[\infty]=F_0^{n-i}F_1^iE(V/\mfrak{m})\otimes \widehat{M}\oplus\bigoplus_{i=1}^nP_i$. Since $-\otimes \widehat{M}\oplus\bigoplus_{i=1}^nP^i$ commutes with direct limits, $P^i[\infty]$ is the direct limit of
\[P^i\hookrightarrow P^i_2\hookrightarrow P^i_3\hookrightarrow\ldots .\]

\begin{cor}\label{generalisedtubesclosure}
Let $\mcal{T}:=((M_i,\phi_i,\psi_i)_{i\in\N_0},(P^i,\alpha^i)_{i=1}^n)$ be a generalised ray tube. Suppose that $N$ is an infinite-dimensional module in the Ziegler closure of $\mcal{T}$.
\begin{enumerate}
\item $N$ is a direct summand of $M[\infty]$, $P^1[\infty],\ldots, P^n[\infty]$, $G$ or $\widehat{M}$.
\item $N$ has $m$-dimension $\leq 1$
\item $M[\infty]$ and $P^i[\infty]$ are $\Sigma$-pure-injective
\item $\widehat{M}$ has acc on pp-definable subgroups
\item A module in the Ziegler closure of $\widehat{M}$ is either a direct summand of $G$ or $\widehat{M}$
\item A module in the Ziegler closure of $P^i[\infty]$ (respectively $M[\infty]$) is either a direct summand of $G$ or $P^i[\infty]$ (respectively $M[\infty]$).
\end{enumerate}
\end{cor}
\begin{proof}
(1) Follows from \ref{intfunmappingontogenraytube} and \ref{indoveronepointextensions}.

(2)Interpretation functors don't increase m-dimension plus \ref{indoveronepointextensions}.

(3) and (4) Interpretation functors preserve acc and dcc on pp-definable subgroups.

(5) and (6) It follows from \ref{indoveronepointextensions} that $\{F_0^nQ, N\}$ is a closed subset for any infinite-dimensional indecomposable pure-injective module over $R_n$. Thus $\{M\in\Zg_\mcal{A} | \ M|G \text{ or } M|N\otimes\widehat{M}\oplus \bigoplus_{i=1}^nP^i \}$ is the closure of $N\otimes\widehat{M}\oplus \bigoplus_{i=1}^nP^i$ for any infinite-dimensional indecomposable pure-injective $N$ over $R_n$.
\end{proof}

\section{Short embeddings}\label{shortembeddings}

In this section we will introduce a special class of embeddings, and then use this notion to investigate the endomorphism rings of direct limits along rays in ray and coray tubes.

\begin{definition}
Let $M,N$ be finitely presented modules. We call an embedding $i:M\hookrightarrow N$ \textbf{short} if there exists a natural number $n\in\N$ such that for all $\overline{a}\in M$, if $\phi$ generates the pp-type of $\overline{a}$ in $M$ and $\psi$ generates the pp-type of $i\overline{a}$ in $N$ then the interval $[\psi,\phi]$ is of length $\leq n$.
\end{definition}

It is tempting to believe that all irreducible embeddings are short. The following lemma and example shows that this is not the case.

\begin{lemma}
For $i\in\N$, let $M_i,N_i\in\mod-R$ be indecomposable, $g_i:M_i\rightarrow N_i$, $f_i:M_i\rightarrow M_{i+1}$ and $h_i:N_i\rightarrow N_{i+1}$.

If
\[\xymatrix@C=1cm{
  0 \ar[r] & M_i \ar[rr]^{\scriptsize{\left(
                            \begin{array}{c}
                              g_i \\
                              f_i \\
                            \end{array}
                          \right)}
  } && N_i\oplus M_{i+1} \ar[rr]^{\scriptsize{\left(
                                    \begin{array}{cc}
                                      h_i & g_{i+1} \\
                                    \end{array}
                                  \right)}} && N_{i+1} \ar[r] & 0 }\] are almost split exact
then $g_i$ is not short for all $i\in\N$.
\end{lemma}
\begin{proof}
It is enough to show that $g_1$ is not short. Let $\overline{a}$ generate $M_1$, $\phi_n$ generate the pp-type of $f_{n-1}\circ\cdots\circ f_1(\overline{a})$ and $\psi_n$ generate the pp-type of $h_{n-1}\circ\cdots\circ h_1\circ g_1(\overline{a})$. Then, since $\overline{a}$ generates $M_1$, by \cite[Lemma 1.2.28]{PSL}, $\psi_n=\psi_1\wedge\phi_n$. Thus $(\phi_{n+1}+\psi_1)\wedge\phi_n=\phi_{n+1}+\psi_n$. Since
\[\xymatrix@C=1cm{
  0 \ar[r] & M_n \ar[rr]^{\scriptsize{\left(
                            \begin{array}{c}
                              g_n \\
                              f_n \\
                            \end{array}
                          \right)}
  } && N_n\oplus M_{n+1} \ar[rr]^{\scriptsize{\left(
                                    \begin{array}{cc}
                                      h_n & g_{n+1} \\
                                    \end{array}
                                  \right)}} && N_{n+1} \ar[r] & 0 }\] is almost split exact and $g_n\circ f_{n-1}\circ\cdots\circ f_1(\overline{a})=h_n\circ\cdots\circ h_1\circ g_1(\overline{a})$, $\phi_n>\phi_{n+1}+\psi_n$. So \[(\phi_n+\psi_1)\wedge\phi_n=\phi_n>\phi_{n+1}+\psi_n=\phi_{n+1}+\psi_1\wedge\phi_n=(\phi_{n+1}+\psi_1)\wedge\phi_n.\]

  Thus $(\phi_n+\psi_1)\wedge\phi_n>(\phi_{n+1}+\psi_1)\wedge\phi_n$. Hence $\phi_n+\psi_1>\phi_{n+1}+\psi_1$.

  So \[\phi_1+\psi_1>\phi_2+\psi_1>\phi_3+\psi_1\ldots\] is a strictly decreasing sequence of pp-formulas in the interval $[\psi_1,\phi_1]$.

 \end{proof}

\begin{example}

Consider the path algebra of the quiver $\xymatrix{1 \ar@/^/[r]^\alpha \ar@/_/[r]_\beta & 2}$.

The above lemma shows that both irreducible embeddings $f,g$ of the projective $P(2)$ at vertex $2$ into the projective $P(1)$ at vertex $1$ are not short.

The pp-type of any non-zero element $m$ in $P(2)$ is generated by $e_2|x$. The pp-type of $f(m)$ is generated by $\exists y\ x=y\alpha$ and the pp-type of $g(m)$ is generated by $\exists  y \ x=y\beta$. We have the following infinite strictly descending chain of pp-$1$-formulas between $e_2|x$ and $\exists \ y x=y\beta$.
\[e_2|x>\exists y \ x=y\alpha+\exists y \ x=y\beta>\exists y_1, y_2 \ x=y_1\alpha\wedge y_1\beta=y_2\alpha+\exists y \ x=y\beta\]
\[>\exists y_1, y_2, y_3 \ x=y_1\alpha\wedge y_1\beta=y_2\alpha\wedge y_2\beta=y_3\alpha+\exists y \ x=y\beta>\ldots\]

\end{example}

\begin{remark}
If $f:M\rightarrow N$ is a short embedding and $f=gh$ then $h$ is a short embedding.
\end{remark}

\begin{lemma}
An embedding $f:M\hookrightarrow N$ is a short embedding if and only if the cokernel, $F\in (\mod-R,\Ab)^{\text{fp}}$, of
\[(N,-)\xrightarrow{(f,-)} (M,-)\rightarrow F\rightarrow 0\] is finite-length.
\end{lemma}
\begin{proof}
Suppose that $f$ is a short embedding. Let $\overline{a}$ generate $M$ and $\overline{b}:=f(\overline{a})$. Let $\phi$ generate the pp-type of $\overline{a}$ in $M$ and $\psi$ generate the pp-type of $\overline{b}$ in $N$. Then, by \cite[1.2.18]{PSL} and \cite[proof of 10.2.30]{PSL}, $F\cong F_{\phi/\psi}$, which is finite-length since $f$ is short.

Conversely, suppose that $F$ is finite-length. Let $\overline{a}\in M$ and $\overline{b}:=f(\overline{a})$. Let $\phi$ generate the pp-type of $\overline{a}$ in $M$ and $\psi$ generate the pp-type of $\overline{b}$ in $N$.

Then

\xymatrix@=20pt{
  (N/\langle \overline{b}\rangle,-) \ar[d] \ar[r]^{(f',-)} &   (M/\langle \overline{a}\rangle,-) \ar[d] &   &  \\
  (N,-) \ar[d] \ar[r]^{(f,-)} & (M,-) \ar[d] \ar[r] & F \ar[d] \ar[r] & 0  \\
  F_{\psi} \ar[d] \ar[r] & F_{\phi} \ar[d] \ar[r] & F_{\phi/\psi}  \ar[r] & 0  \\
  0 & 0 &   &   }

So $F_{\phi/\psi}$ is an epimorphic image of $F$ and hence finite-length. Thus $[\psi,\phi]$ is finite-length.

\end{proof}

\begin{remark}
An embedding $f:M\rightarrow N$ is short if the morphism $(f,-)\in (\mod-R,\Ab)^{\text{fp}}$ is an epimorphism after localising at the serre subcategory of finite-length objects.
\end{remark}

\begin{remark}
Let $\xymatrix@C=0.5cm{
  0 \ar[r] & A \ar[rr]^{f} && B \ar[rr]^{g} && C \ar[r] & 0 }$ be almost split exact.
After localising at the Serre subcategory of finite-length objects, the sequence
\[\xymatrix@C=0.5cm{
  0 \ar[r] & (C,-) \ar[rr]^{(g,-)} && (B,-) \ar[rr]^{(f,-)} && (A,-) \ar[r] & 0 }\]
is a short exact sequence.
\end{remark}

\begin{lemma}\label{shorttransfersinglealmostsplit}
Suppose that $i:A\rightarrow B_1$ is an embedding, $\mu:B_2\rightarrow C$ is a short embedding (of length $n$) and
\[\xymatrix@C=1cm{
  0 \ar[r] & A \ar[rr]^{\scriptsize{\left(
                          \begin{array}{c}
                            i \\
                            \pi \\
                          \end{array}
                        \right)}
  } && B_1\oplus B_2 \ar[rr]^{\scriptsize{\left(
                                \begin{array}{cc}
                                  \gamma & \mu \\
                                \end{array}
                              \right)}
  } && C \ar[r] & 0 }\] is almost split exact. Then $i$ is short (of length $\leq n+1$) .
\end{lemma}

\begin{proof}
Since \[\xymatrix@C=1cm{
  0 \ar[r] & A \ar[rr]^{\scriptsize{\left(
                          \begin{array}{c}
                            i \\
                            \pi \\
                          \end{array}
                        \right)}
  } && B_1\oplus B_2 \ar[rr]^{\scriptsize{\left(
                                \begin{array}{cc}
                                  \gamma & \mu \\
                                \end{array}
                              \right)}
  } && C \ar[r] & 0 }\] is almost split exact, the cokernel $F\in (\mod-R,\Ab)^{\text{fp}}$ of
  \[(B_1\oplus B_2,-)\xrightarrow{\scriptsize{(\left(
                          \begin{array}{c}
                            i \\
                            \pi \\
                          \end{array}
                        \right),-)}}(A,-)\rightarrow F\rightarrow 0\] is simple. Let $\mcal{S}_0$ be the serre subcategory of $(\mod-R,\Ab)^{\text{fp}}$ generated by the simple functors.

Thus, since serre localisation is exact and the contravariant Yoneda embedding of $\mod-R$ into $(\mod-R,\Ab)^{\text{fp}}$ is left-exact,
\[\xymatrix@C=1cm{
  0 \ar[r] & (C,-) \ar[rr]^{\scriptsize{(\left(
                                \begin{array}{cc}
                                  \gamma & \mu \\
                                \end{array}
                              \right),-)}
  } && B_1\oplus B_2 \ar[rr]^{\scriptsize{(\left(
                          \begin{array}{c}
                            i \\
                            \pi \\
                          \end{array}
                        \right),-)}
  } && (A,-) \ar[r] & 0 }\] is exact in $(\mod-R,\Ab)^{\text{fp}}/\mcal{S}_0$. Thus the following diagram is both a pushout and a pullback in $(\mod-R,\Ab)^{\text{fp}}/\mcal{S}_0$.
\[\xymatrix@=20pt{
  (C,-) \ar[d]_{(\mu,-)} \ar[r]^{(\gamma,-)} & (B_1,-) \ar[d]^{(i,-)} \\
  (B_2,-) \ar[r]^{(\pi,-)} & (A,-)   }\]
Since $\mu:B_2\rightarrow C$ is a short embedding, $(\mu,-)$ is an epimorphism in $(\mod-R,\Ab)^{\text{fp}}/\mcal{S}_0$. Thus $(i,-)$ is also an epimorphism in $(\mod-R,\Ab)^{\text{fp}}/\mcal{S}_0$ and hence $i$ is a short embedding.

\end{proof}

\begin{lemma}\label{generalisedshorttransfersinglealmostsplit}
Suppose that $L_1,\ldots,L_{n+1}$ and $M_1,\ldots,M_{n+1}$ are indecomposable finite-dimensional modules and that for $1\leq j\leq n$
\[\xymatrix@C=1cm{
  0 \ar[r] & M_j \ar[rr]^{\scriptsize{\left(
                          \begin{array}{c}
                            i_j \\
                            \pi_j \\
                          \end{array}
                        \right)}} && M_{j+1}\oplus L_j \ar[rr]^{\scriptsize{\left(
                                \begin{array}{cc}
                                  \pi_{j+1} & \mu_j \\
                                \end{array}
                              \right)}} && L_{j+1} \ar[r] & 0 }\] is almost spilt exact.

If $\mu_n\circ\ldots\circ\mu_1$ is short then $i_n\circ\ldots\circ i_1$ is short.
\end{lemma}

\begin{proof}
We need to show that if $(\mu_n\circ\ldots\circ\mu_1,-)$ is an epimorphism in $(\mod-R,\Ab)^{\text{fp}}/\mcal{S}_0$ then $(i_n\circ\ldots\circ i_1,-)$ is an epimorphism in $(\mod-R,\Ab)^{\text{fp}}/\mcal{S}_0$. Since each sequence in the lemma is almost split exact,

\[\xymatrix@R=20pt@C=40pt{
  (L_{j+1},-) \ar[d]_{(\mu_j,-)} \ar[r]^{(\pi_{j+1},-)} & (M_{j+1},-) \ar[d]^{(i_j,-)} \\
  (L_j,-) \ar[r]^{(\pi_j,-)} & (M_j,-)   }\] is a pushout and a pullback square in $(\mod-R,\Ab)^{\text{fp}}/\mcal{S}_0$ for all $1\leq j\leq n$. Putting these pushout and pullback squares along side each other, we get that
 \[\xymatrix@R=20pt@C=40pt{
  (L_{n+1},-) \ar[d]_{(\mu_n\circ\ldots\circ\mu_1,-)} \ar[r]^{(\pi_{n+1},-)} & (M_{n+1},-) \ar[d]^{(i_n\circ\ldots\circ i_1,-)} \\
  (L_1,-) \ar[r]^{(\pi_1,-)} & (M_1,-)   }\]is a pushout and a pullback square in $(\mod-R,\Ab)^{\text{fp}}/\mcal{S}_0$. Thus if $(\mu_n\circ\ldots\circ\mu_1,-)$ is an epimorphism in $(\mod-R,\Ab)^{\text{fp}}/\mcal{S}_0$ then $(i_n\circ\ldots\circ i_1,-)$ is an epimorphism in $(\mod-R,\Ab)^{\text{fp}}/\mcal{S}_0$.
\end{proof}

For an artin algebra $R$ and $A,B\in\mod-R$, the \textbf{radical} of $\mod-R$ is the two-sided ideal defined by
\[\rad(A,B):=\{f\in \Hom(A,B) \st 1_A-g\circ f \text{  is an isomorphism for any } g\in\Hom(B,A) \}.\]

For any natural number $n>1$, $\rad^n(A,B)$ is the set of sums of morphisms $g\circ h$ where $C\in\mod-R$, $g\in \rad(C,B)$ and $h\in \rad^{n-1}(A,C)$.

The $\omega$-radical of $\mod-R$ is the two-sided ideal defined by
\[\rad^\omega(A,B):=\cap_{i\in\N}\rad^i(A,B).\]

We will show that short embeddings are not in the $\omega$-radical.

\begin{lemma}\label{radicalmapsincreasepptype}
Let $R$ be an artin algebra, $A,B\in \mod-R$ and $f\in\Hom(A,B)$. The following statements are equivalent.
\begin{enumerate}
\item $f\in\rad(A,B)$
\item For all non-zero $a\in A$, if $\phi$ generates the pp-type of a non-zero element $a\in A$ and $\psi$ generates the pp-type of $f(a)$ in $B$ then $\psi<\phi$.

\end{enumerate}
\end{lemma}
\begin{proof}
(1)$\Rightarrow$(2) Suppose that $f\in \rad(A,B)$ and $a\in A$ is non-zero. Suppose for a contradiction that $\phi=\psi$. Let $A'\oplus H(a)=A$ where $H(a)$ is the pure-injective hull of $a$ in $A$, see \cite[pg 153]{PSL}. Since $f\in \rad(A,B)$, $f'=f|_{H(a)}\in\rad(H(a),B)$. By \cite[4.3.42]{PSL}, since $\text{pp}^{H(a)}(a)=\text{pp}^B(f'(a))$, $f'$ is a pure-embedding. Since $H(a)$ is pure-injective, $f'$ is split. Thus $f'\notin \rad(A,B)$.

(2)$\Rightarrow$(1) Suppose that $f\notin \rad(A,B)$. Then there exists $A'$ a direct summand of $A$ such that $A'$ is indecomposable and $f|_{A'}=f':A'\rightarrow B$ is not in $\rad(A',B)$. Let $g\in\Hom(B,A')$. Take $a\in A'$ such that $gf'(a)\neq 0$ if such an element exists. The pp-type of $(1_{A'}-g\circ f')(a)$ is equal to the pp-type of $a$ since $f'$ strictly increases the pp-types of $a$. Since $A'$ is indecomposable, this implies, by \cite[4.3.45]{PSL}, that $1_{A'}-g\circ f'$ is an isomorphism. Thus $f'\in\rad(A',B)$.
\end{proof}

The following corollary implies that short-embeddings are not in the $\omega$-radical.

\begin{cor}\label{omegaradincreasespptype}
Suppose $X,Y\in\mod-R$ and $f:X\rightarrow Y\in \rad^\omega(X,Y)$. If $a\in X$ is non-zero, $\phi$ generates the pp-type of $a$ in $X$ and $\psi$ generates the pp-type of $f(a)$ in $Y$ then either $f(a)=0$ or $[\psi,\phi]$ is infinite-length.
\end{cor}

\begin{lemma}\label{descriptionofradworadomega}
Let $A,B$ be indecomposable and $f\in \rad(A,B)\backslash\rad^\omega(A,B)$. Then $f=u_1+u_2+\ldots+u_m+v$ where each $u_i$ is either zero or a sum of compositions of $i$ irreducible maps between indecomposable modules and $v\in \rad^\omega(A,B)$.
\end{lemma}
\begin{proof}
Let $A,B$ be indecomposable and $f\in \rad(A,B)\backslash\rad^\omega(A,B)$. Since $f\notin \rad^\omega(A,B)$, there exists $n\in\N$ such that $f\in\rad^n(A,B)\backslash\rad^{n+1}(A,B)$. By \cite[Proposition 7.4 (ii)]{ARS}, there exists $u,v\in\Hom(A,B)$ such that $f=u+v$, $u$ is non-zero and a sum of compositions of $n$ irreducible maps between indecomposable modules and $v\in \rad^{n+1}(A,B)$.

By iterating this process we get $f=u_1+u_2+\ldots+u_{m-1}+v$ where each $u_i$ is either zero or a sum of compositions of $i$ irreducible maps between indecomposable modules and $v\in \rad^m(A,B)$. By \cite[Lemma 7.2]{ARS} there exists an $m\in\N$ such that $\rad^\omega(A,B)=\rad^m(A,B)$.

\end{proof}

In the proof of the next two statements, we will freely use the following construction.
Let $M$ be a direct limit of finite-dimensional modules $(M_i)_{i\in\N}$ along a chain of monomorphisms $\mu_i:M_i\rightarrow M_{i+1}$. For each $i\in\N$, let $u_i:M_i\rightarrow M$ be the canonical embedding of $M_i$ into $M$. For each $i<j$, let $\gamma_{i,j}:=\mu_{j-1}\ldots\mu_{i+1}\mu_i$.

Any sequence of morphisms $g_i:M_i\rightarrow M_{n_i}$ such that
\[\xymatrix@=40pt{
  M_i \ar[d]_{g_i} \ar[r]^{\mu_i} & M_{i+1} \ar[d]^{g_{i+1}} \\
  M_{n_i} \ar[r]^{\gamma_{n_i,n_{i+1}}} & M_{n_{i+1}}   }\] commutes gives rise to an endomorphism $g$ of $M$ by setting $g(a)=u_i g_i(a)$ whenever $a\in M_i$.

Further, given an endomorphism $g:M\rightarrow M$, let $M_{n_i}$ be such that $\im g|_{M_i}\subseteq M_{n_i}$, $n_i<n_{i+1}$ and $g_i:M_i\rightarrow \im g|_{M_i}\hookrightarrow M_{n_i}$ the homomorphism induced by $g$. The sequence of morphisms $g_i$ is such that
\[\xymatrix@=40pt{
  M_i \ar[d]_{g_i} \ar[r]^{\mu_i} & M_{i+1} \ar[d]^{g_{i+1}} \\
  M_{n_i} \ar[r]^{\gamma_{n_i,n_{i+1}}} & M_{n_{i+1}}   }\] commutes and induces $g$ on $M$.

\begin{proposition}\label{indecomposabledirectlimits}

Suppose that
\[M_1\xrightarrow{\mu_1} M_2\xrightarrow{\mu_2} M_3\xrightarrow{\mu_3} \ldots\]
 is a ray of monomorphisms. For all $k\leq n$, let $\gamma_{k,n}=\mu_k\circ\mu_{k+1}\circ\ldots\circ\mu_{n-1}$. Suppose that for each $j\geq i>1$, every composition of irreducible morphisms from $f:M_i\rightarrow M_k$ factors as $f=\gamma_{i,j}\circ g$ where $g$ is a composition of irreducible maps from $M_i$ to itself.  Further suppose that all $\mu_i$ are short monomorphisms. Then $M:=\bigcup_i M_i$ is indecomposable.
\end{proposition}

\begin{proof}
 Suppose that $e\in \End(M)$ is a non-zero idempotent. Take $a\in \ker(e-1)$. There exists $i\in \N$ such that $a\in M_i$. Let $k\in\N$ be such that $e(M_i)\subseteq M_k$ and $k>i$. Let $f:M_i\rightarrow M_k$ be the map induced by $e$. So $f(a)=\gamma_{i,k}(a)$. Since $\gamma_{i,k}$ is a short embedding, by \ref{omegaradincreasespptype}, $f\notin \rad^{\omega}(M_i,M_k)$.

We now show that $f$ is an embedding. By \ref{descriptionofradworadomega}, $f=\alpha\gamma_{i,k}+\gamma_{i,k}g_1+\ldots\gamma_{i,k}g_n+v$ where $\alpha\in k$, $v\in\rad^\omega$ and $g_1,\ldots, g_n$ are compositions of irreducible maps from $M_i$ to $M_i$.

Suppose for a contradiction that $\alpha=0$. By definition of $a$, \[(\gamma_{i,k}g_1+\ldots\gamma_{i,k}g_n+v)(a)=\gamma_{i,k}(a).\] So
\[\gamma_{i,k}(-1_{M_i}+g_1+\ldots+g_n)(a)+v(a)=0.\]
Since $g_1+\ldots+g_n$ is a non-isomorphism, $-1_{M_i}+g_1+\ldots+g_n$ is an isomorphism. So, the pp-type of $(-1_{M_i}+g_1+\ldots+g_n)(a)$ is the same as that of $a$. Since $\gamma_{i,k}$ is short and $v\in\rad^\omega(M_i,M_k)$, the pp-type of $\gamma_{i,k}(-1_{M_i}+\alpha_1g_1+\ldots\alpha_ng_n )(a)$ is not equal to the pp-type of $-v(a)$. Thus we have our contradiction and $\alpha\neq 0$.

Suppose that $b\in M_i$ and $f(b)=0$. Let $\phi$ generate the pp-type of $b$ in $M_i$. The pp-type of $(\alpha1_{M_i}+g_1+\ldots+g_n)(b)$ is equal to the pp-type of $b$. Thus $[\phi,\psi_1]$ is finite-length where $\psi_1$ generates the pp-type of $\gamma_{i,k}(\alpha1_{M_i}+g_1+\ldots+g_n)(b)$ in $M_k$. Since $v\in \rad^\omega$, either $v(b)=0$ or $[\phi,\psi_2]$ is of infinite-length where $\psi_2$ generates the pp-type of $-v(b)$ in $M_k$. But $f(b)=0$ implies that the pp-type of $\gamma_{i,k}(\alpha1_{M_i}+g_1+\ldots+g_n)(b)$ in $M_k$ is equal to the pp-type of $-v(b)$ in $M_k$. Thus $f$ is an embedding.

Thus, for all $j\geq i$, the map $f_j:M_j\rightarrow M_{n_j}$ induced by $e$ is an embedding. Hence $e$ is an embedding. Thus $M$ is indecomposable.

\end{proof}

\begin{remark}
This is enough to show that a direct limit along a ray in an $\Z A_\infty$ Auslander-Reiten component is indecomposable.
\end{remark}

\begin{proposition}\label{topofendo}
Suppose that
\[M_1\xrightarrow{\mu_1} M_2\xrightarrow{\mu_2} M_3\xrightarrow{\mu_3} \ldots\]
 is a ray of monomorphisms. For all $k\leq n$, let $\gamma_{k,n}=\mu_k\circ\mu_{k+1}\circ\ldots\circ\mu_{n-1}$. Suppose that for each $i\in\N$ there exists a nilpotent endomorphism $\phi_i:M_i\rightarrow M_i$ such that $\mu_i\phi_i=\phi_{i+1}\mu_i$ and for each $j\geq i>1$, every non-zero composition of irreducible morphisms from $f:M_i\rightarrow M_k$ factors as $f=\gamma_{i,j}\circ (\phi_i)^p$. Further, suppose that each $\mu_i$ is a short embedding.

 Then the canonical embedding of $k$ into $\End(\bigcup_iM_i)/\rad\End(\bigcup_iM_i)$ is an isomorphism.
\end{proposition}
\begin{proof}
Any homomorphism $f:M_i\rightarrow M_j$ is of the form
\[\sum_{p=0}^{l}\alpha_p\gamma_{i,j}(\phi_i)^p+\delta\] where $\alpha_p\in k$ and $\delta\in \rad^\omega(M_i,M_j)$.

Suppose for a contradiction that $f$ is a short embedding and $\alpha_0=0$. Let $m\in\ker \phi_i$. Then $f(m)=\delta(m)$. Let $\phi$ generate the pp-type of $m$ in $M_i$ and $\psi$ generate the pp-type of $\delta(m)$ in $M_j$. Since $\delta\in \rad^\omega(M_i,M_j)$ either $\delta(m)=0$ or $[\psi,\phi]$ is infinite-length. In either case this contradicts the fact that $f$ is a short embedding. Thus all short embeddings $f:M_i\rightarrow M_j$ are of the form
\[\sum_{p=0}^{l}\alpha_p\gamma_{i,j}(\phi_i)^p+\delta\] where $\alpha_p\in k$, $\alpha_0\neq 0$ and $\delta\in \rad^\omega(M_i,M_j)$.

We now show that if $g:\bigcup_i M_i\rightarrow \bigcup_i M_i$ is an isomorphism then the induced map $g_i:M_i\rightarrow M_{n_i}$ is a short embedding. Take $a\in M_i$ and let $\psi$ generate the pp-type of $g_i(a)$ in $M_{n_i}$ and $\phi_j$ generate the pp-type of $\gamma_{i,j}(a)$ in $M_j$ for all $j\geq i$. So $\psi$ is in the pp-type of $g(a)$ in $\bigcup_iM_i$ and since $g$ is an isomorphism, $\psi$ is in the pp-type of $a$ in $\bigcup_iM_i$. Thus $\psi$ is in the pp-type of $a$ in $M_j$ for for some $j\geq i$ and thus $\psi\geq \phi_j$ for some $j\geq i$. Since each $\mu_i$ is short, $[\phi_j,\phi_i]$ is finite-length and thus so is $[\psi,\phi_i]$. Thus $g_i$ is short.

Combining the previous two paragraphs, we can see that the sum of two non-isomorphisms of $\bigcup_iM_i$ is a non-isomorphism. Thus $\End(\bigcup_iM_i)$ is local.

Finally, suppose that we have a sequence of morphism $g_i:M_i\rightarrow M_{n_i}$ such that
\[ \xymatrix@=40pt{
  M_i \ar[d]_{g_i} \ar[r]^{\mu_i} & M_{i+1} \ar[d]^{g_{i+1}} \\
  M_{n_i} \ar[r]^{\gamma_{n_i,n_{i+1}}} & M_{n_{i+1}}   },\]
\[g_i:=\sum_{p=0}^{l_i} \alpha_p^i\gamma_{i,n_i}(\phi_i)^p+\delta_i\] and let $g:\bigcup_iM_i\rightarrow \bigcup_i M_i$ be the morphism induced by the $g_i$s.

Then
\[\alpha_0^i\gamma_{i,n_{i+1}}-\alpha_0^{i+1}\gamma_{i,n_{i+1}}=\]\[(\sum_{p=1}^{l_i} \alpha_p^i\gamma_{i+1,n_{i+1}}(\phi_{i+1})^p+\delta_{i+1})\gamma_{i+1,n_{i+1}}-\gamma_{n_i,n_j}(\sum_{p=1}^{l_i} \alpha_p^i\gamma_{i,n_i}(\phi_i)^p+\delta_i).\]
Take $a\in \ker \phi_i$. Then
\[(\alpha_0^i-\alpha_0^{i+1})\gamma_{i,n_{i+1}}(a)=\delta_{i+1}\gamma_{i+1,n_{i+1}}(a)- \gamma_{n_i,n_j}\delta_i(a)\] since $\mu_i\phi_i=\phi_{i+1}\mu_i$ for all $i\in\N$. But then $\gamma_{i,n_{i+1}}$ is a short embedding and $\delta_{i+1}\gamma_{i+1,n_{i+1}}- \gamma_{n_i,n_j}\delta_i\in\rad^\omega(M_i,M_{n_{i+1}})$, so $\alpha^i_0-\alpha^{i+1}_0=0$. Let $\alpha_0=\alpha_0^i$ for all $i\in\N$.

Thus $g=\alpha_0\text{Id} + h$ where $h:\bigcup_i M_i\rightarrow \bigcup_i M_i$ is the morphism induced by the sequence $h_i:M_i\rightarrow M_{n_i}$
\[h_i:=\sum_{p=1}^{l_i} \alpha_p^i\gamma_{i,n_i}(\phi_i)^p+\delta_i.\] Thus the canonical embedding of $k$ into $\End(\bigcup_iM_i)/\rad\End(\bigcup_iM_i)$ is an isomorphism.

\end{proof}

\section{Ray and coray tubes for finite-dimensional algebras}\label{raycorayforfdalg}

In this section we will describe the short embeddings in ray and coray tubes. We will then use \ref{topofendo} to show that the direct limit along a ray in either a ray or coray tube is indecomposable and that the canonical embedding of $k$ into its endomorphism ring factored out by the radical is an isomorphism.

We give a slightly different definition of a ray tube to that given in the literature (for instance in \cite{coherenttubes} or \cite{Luishapes}) but it is equivalent.


For each tuple of integers $(m;n_0,\ldots,n_{m-1})$ let $Q(m;n_0,\ldots,n_{m-1})$ be the translation quiver with points $S_i^k[j]$ where $0\leq i <m$, $0\leq k\leq n_i$ and $j\in\N$. The index $i$ will always be read $\textbf{mod} \ m$. The arrows in $Q(m;n_1,\ldots,n_m)$ are
\[\mu_i^k[j]:S_i^k[j]\rightarrow S_i^k[j+1],\]
\[\lambda_i^{n_i}[j+1]:S_i^{n_i}[j+1]\rightarrow S_{i+1}^0[j]\] and, when $n_i>k\geq 0$
\[\lambda_i^k[j]:S_i^k[j]\rightarrow S_i^{k+1}[j].\]

Thus for $k\neq n_i$ we have
\[\lambda_i^k[j+1]\circ\mu_i^k[j]=\mu_i^{k+1}[j]\circ\lambda_i^k[j] ,\] for $k=n_i$ and $j\neq 1$ we have
\[\lambda_i^{n_i}[j+1]\circ\mu_i^{n_i}[j]=\mu_{i+1}^{0}[j]\circ\lambda_i^{n_i}[j+2]\] and finally for $k=n_i$ and $j=1$ we have
\[ \lambda_i^{n_i}[2]\circ\mu_i^{n_i}[1]=0.\]

We refer to these relations as mesh relations. We call such a translation quiver a \textbf{ray tube}. We call the dual translation quiver a \textbf{coray tube}.

We now list some definitions and elementary facts following from the mesh relations.

\begin{enumerate}
\item Any non-zero composition of irreducible morphisms can be written as a sequence of $\mu_i^k[j]$ followed by a sequence of $\lambda_i^k[j]$ and also as a sequence of $\lambda_i^k[j]$ followed by a sequence of $\mu_i^k[j]$.
\item  If $l>j$ let $\mu_i^k[j\rightarrow l]:S_i^k[j]\rightarrow S_i^k[l]$ be \[\mu_i^k[l-1]\circ\cdots\circ\mu_i^k[j].\]
\item Let $\lambda_i^k[j+ m\rightarrow j]:S_i^k[j+m]\rightarrow S_i^k[j]$ be \[(\lambda_{i+m}^{k-1}[j]\circ\cdots\circ\lambda^0_{i+m}[j])\circ\cdots\cdots\circ(\lambda_{i+1}^{n_{i+1}}[j+(m-1)]\circ\cdots\]
\[\cdots\circ\lambda^0_{i+1}[j+(m-1)])\circ(\lambda_i^{n_i}[j+m]\circ\cdots\circ\lambda_i^{k+1}[j+m]\circ\lambda_i^k[j+m]).\]
If $\gamma>1$ let
\[\lambda_i^k[j+\gamma m\rightarrow j]:=\lambda_i^k[j+\gamma m\rightarrow j+(\gamma-1)m]\circ\cdots\circ\lambda_i^k[j+ m\rightarrow j].\]
 Note that the relations imply that for all $j\in\N$
\[\mu_i^k[j\rightarrow j+m]\circ\lambda_i^k[j+m\rightarrow j]=\lambda_i^k[j+2m]\circ\mu_i^k[j+m\rightarrow j+2m]\]
\item For all $j\in\N$,
\[\phi_i^k[j+m]:=\mu_i^k[j\rightarrow j+m]\circ\lambda_i^k[j+m\rightarrow j].\]

Note that the mesh relations imply that
\[(\phi_i^k[j+m])^n:=\mu_i^k[j\rightarrow j+nm]\circ\lambda_i^k[j+nm\rightarrow j]\]
\item \label{morphfac} Every non-zero composition of irreducible morphisms $S^k_i[j]\rightarrow S^k_i[l]$ where $l\geq j$ is of the form
\[\alpha \mu_i^k[j\rightarrow l]\phi_i^k[j]^p\] where $j>pm$ and $\alpha\in k$.

\item \label{morphfacdual} Every non-zero composition of irreducible morphisms $S_i^k[1+(\gamma+1) m]\rightarrow S_i^k[1+\gamma m]$ where $\gamma\in \N_0$ is of the form
\[\alpha\phi_i^k[1+\gamma m]^p\lambda_i^k[1+(\gamma+1) m\rightarrow 1+\gamma m]\] where $p\leq \gamma+1$ and $\alpha\in k$.
\end{enumerate}

%

\begin{lemma}\label{muareshort}
The morphisms $\mu_i^k[j]$ are short embeddings
\end{lemma}
\begin{proof}
For any $0\leq i<m$, $\mu_i^{n_i}[1]$ is a short embedding since
\[\xymatrix@C=1cm{
  0 \ar[r] & S_i^{n_i}[1] \ar[rr]^{\tiny{\mu_i^{n_i}[1]}} && S_i^{n_i}[2] \ar[rr]^{\tiny{\lambda_i^{n_i}[2]}} && S_{i+1}^{0}[1]  \ar[r] & 0 }\] is almost split exact.

Since
\[\xymatrix@C=0.9cm{0 \ar[r] & S_i^k[j] \ar[rr]^{\scalebox{0.7}{$\begin{pmat}  \mu_i^k[j] \\ \lambda_i^k[j] \\ \end{pmat}$ \ \ \ \ }} && S_i^k[j+1]\oplus S_i^k[j] \ar[rr]^{\scalebox{0.7}{ \ \ \ $\begin{pmat}
                      \lambda_i^k[j+1] & \mu^{k+1}_i[j] \\
                    \end{pmat}$}} && S^{k+1}_i[j+1] \ar[r] & 0 }\]
is almost split exact, repeated applications of lemma \ref{shorttransfersinglealmostsplit} implies that if $\mu_i^{n_i}[j]$ is a short embedding then $\mu_i^k[j]$ is a short embedding for all $0\leq k\leq n_i$.

Since
\[\xymatrix@C=0.7cm{
  0 \ar[r] & {S_i^{n_i}[j+1]} \ar[rr]^{\scalebox{0.7}{$\begin{pmat} \mu_i^{n_i}[j+1]  \\  \lambda_i^{n_i}[j+1]\\ \end{pmat}$ \ \ \ \ \ \ \ }
  } && S_i^{n_i}[j+2]\oplus S_{i+1}^0[j] \ar[rr]^{ \ \ \ \ \scalebox{0.7}{ \ $\begin{pmat}\lambda_i^{n_i}[j+2] & \mu_{i+1}^0[j] \\\end{pmat}$}} && S_{i+1}^0[j+1] \ar[r] & 0 }\]
is almost split exact, lemma \ref{shorttransfersinglealmostsplit} implies that if $\mu_{i+1}^0[j]$ is a short embedding then $\mu_i^{n_i}[j+1]$ is a short embedding.

Thus by induction, $\mu_i^k[j]$ is a short embedding for all $0\leq i<m$, $0\leq k\leq n_i$ and $j\in\N$.
\end{proof}

\begin{proposition}\label{directlimitsalongmuisind}
Any direct limit $S_i^k[\infty]$ along the sequence $(\mu_i^k[j])_{j\in \N}$ is indecomposable and the canonical embedding of $k$ into  $\End(S_i^k[\infty])/\rad\End(S_i^k[\infty])$ is an isomorphism.
\end{proposition}
\begin{proof}
This follows from \ref{topofendo} plus \ref{muareshort} plus point \ref{morphfac} from page \pageref{morphfac}.
\end{proof}

\begin{lemma}\label{duallambdashort}
For all $j$, $k$ and $i$, the morphism $\lambda_i^k[j\rightarrow j-m]^*$ is short embeddings.
\end{lemma}
\begin{proof}
The relations show that $\lambda_i^{n_i}[2]^*$ is a short embedding for all $i$.

Suppose that $\lambda_i^{k}[2]^*\ldots\lambda_i^{n_i-1}[2]^*\lambda^{n_i}_i[2]^*$ is a short embedding. Since $\left(
                                                                                                              \begin{array}{c}
                                                                                                                \lambda_i^{k-1}[2]^* \\
                                                                                                                \mu_i^k[1]^* \\
                                                                                                              \end{array}
                                                                                                            \right)
$ is short, $\left(
               \begin{array}{c}
                 \lambda_i^{k-1}[2]^*\circ\ldots\circ\lambda_i^{n_i}[2]^* \\
                 \mu_i^k[1]^*\circ\lambda_i^k[2]^*\circ\ldots\circ\lambda_i^{n_i}[2]^* \\
               \end{array}
             \right)
$ is short. Since $\mu_i^k[1]^*\circ\lambda_i^k[2]^*\circ\ldots\circ\lambda_i^{n_i}[2]=0$,
$\lambda_i^{k-1}[2]^*\circ\ldots\circ\lambda_i^{n_i}[2]^*$ is a short embedding. Thus, by induction, $\lambda_i^0[2]^*\circ\ldots\circ\lambda_i^{n_i}[2]^*$ is a short embedding.

Now applying \ref{generalisedshorttransfersinglealmostsplit}, we get that $\lambda_i^{0}[j]^*\circ\ldots\circ\lambda_i^{n_i}[j]^*$ is a short embedding for all $j\geq 2$.

Thus since a composition of short embedding is a short embedding, $\lambda_i^k[j\rightarrow j-m]^*$ is short for all $j>m$.

\end{proof}

\begin{proposition}\label{directlimitalonglambdadual}
The direct limit $(S_i^k)^*[\infty]$ along a sequence of $(\lambda_i^k[1+mj\rightarrow 1+ m(j-1)]^*)_{j\in\N_0}$ is indecomposable and $\End((S_i^k)^*[\infty])/\rad\End((S_i^k)^*[\infty])=k$.
\end{proposition}

\begin{proof}
This follows from \ref{topofendo} plus \ref{duallambdashort} plus point \ref{morphfacdual} from page \pageref{morphfacdual}.
\end{proof}

\section{Elementary Socles and indecomposable $k$-duals}\label{elsocanddual}

\begin{definition}\cite[10.2]{Herzogduality}
Let $U$ be a $\Sigma$-pure-injective module. Define $\text{elsoc}^0(U)=0$. For each ordinal $\alpha$, let
\[\text{elsoc}^{\alpha+1}(U):=\Sigma\{\phi(U) \st \phi(U)\nsubseteq\text{elsoc}^\alpha(U) \text{ and } \phi(U) \text{ is minimal such } \}+ \text{ elsoc }^{\alpha}(U).\] For limit ordinals, $\lambda$, let
\[\text{ elsoc }^{\lambda}(U):=\bigcup_{\alpha<\lambda}\text{elsoc}^\lambda(U).\]

The elementary socle length of a $\Sigma$-pure-injective modules $U$ is the least ordinal $\alpha$ such that $\text{elsoc}^{\alpha}(U)=U$.
\end{definition}

\begin{definition}
Let $N$ be an indecomposable pure-injective right module over an arbitrary ring $R$. Let $S$ be the (necessarily) local endomorphism ring of $N$ and let $\mfrak{m}$ be its maximal ideal. The local dual of $N$ is defined to be the left module $\Hom_S(N,E(S/\mfrak{m}))$ where $E(S/\mfrak{m})$ is the injective hull of $S/\mfrak{m}$ as an $S$-module.
\end{definition}

\begin{theorem}\cite{HerzogZieglersindcrit}
Let $N$ be an indecomposable $\Sigma$-pure-injective module over an arbitrary ring $R$. The local dual of $N$ is indecomposable.
\end{theorem}

The key idea for the proof of the following proposition was explained to me by Ivo Herzog.

\begin{proposition}\label{kdualisind}
Let $R$ be a $k$-algebra. Let $N$ be an indecomposable $\Sigma$-pure-injective $R$-module whose (local) endomorphism ring $S$ with maximal ideal $\mfrak{m}$ is such that the canonical embedding from $k$ into $S/\mfrak{m}$ is an isomorphism. The local dual of $N$ is isomorphic to the $k$-dual of $N$ and is thus indecomposable.
\end{proposition}
\begin{proof}
Since $R$ is a $k$-algebra, $S$ is also a $k$-algebra. Let $\mu:E(S/\mfrak{m})\rightarrow S/\mfrak{m}=k$ be a $k$-linear map with $S/\mfrak{m}\nsubseteq \ker \mu$. The map from $\Hom_S(N,E(S/\mfrak{m}))$ to $\Hom_k(N,k)$ given by $f\mapsto \overline{f}:=\mu\circ f$ is a left $R$-module embedding. It is an embedding since $S/\mfrak{m}$ is essential in $E(S/\mfrak{m})$ and hence for any $0\neq f\in \Hom_S(N,E(S/\mfrak{m}))$, $\im f\cap S/\mfrak{m}\neq 0$.

We now prove the following two statements:
\begin{enumerate}
\item Let $\alpha$ be an ordinal. If $\nu\in \Hom_k(N,k)$ and $\nu|_{\elsoc^\alpha(N)}=0$ then there exists $f\in \Hom_S(N,E(S/\mfrak{m}))$ such that $(\nu-\overline{f})|_{\elsoc^{\alpha+1}(N)}=0$.
\item Let $\lambda$ be a limit ordinal. If for all $\alpha<\lambda$ there exists $g_\alpha\in \Hom_S(N,E(S/\mfrak{m}))$ such that $(\nu-\overline{g_\alpha})|_{\elsoc^\alpha(N)}=0$ and $(\overline{g_\alpha}-\overline{g_\beta})|_{\elsoc^\alpha(N)}=0$ for $\alpha<\beta<\lambda$ then there exists $g_\lambda\in \Hom_S(N,E(S/\mfrak{m}))$ such that $(\nu-\overline{g_\lambda})|_{\elsoc^\lambda(N)}=0$.
\end{enumerate}

1. Let $\nu\in \Hom_k(N,k)$ be such that $\nu|_{\elsoc^\alpha(N)}=0$. Then $\ker\nu\cap\elsoc^{\alpha+1}(N)$ is an $S$-submodule of $\elsoc^{\alpha+1}(N)$ and $\nu$ induces an $S$-module homomorphism $\nu|_{\elsoc^{\alpha+1}(N)}$ from $\elsoc^{\alpha+1}(N)/\ker\nu\cap\elsoc^{\alpha+1}(N)$ to $k=S/\mfrak{m}$. Post-composing $\nu|_{\elsoc^{\alpha+1}(N)}$ with the embedding of $S/\mfrak{m}$ into $E(S/\mfrak{m})$ we get a morphism $\lambda:\elsoc^\alpha(N)/\ker\nu\cap\elsoc^{\alpha+1}(N)\rightarrow E(S/\mfrak{m})$. Since $E(S/\mfrak{m})$ is injective, $\lambda$ extends to a map from $N/\ker\nu\cap\elsoc^{\alpha+1}(N)\rightarrow E(S/\mfrak{m})$. Pre-composing this map with the projection of $N$ onto $N/\ker\nu\cap\elsoc^{\alpha+1}(N)\rightarrow E(S/\mfrak{m})$, we get an $S$-homomorphism $f$ from $N$ to $E(S/\mfrak{m})$ such that $(\nu-\overline{f})|_{\elsoc^{\alpha+1}(N)}=0$.

2. For each $\alpha<\lambda$, let $\{\phi^\alpha_i\st i\in I_\alpha\}$ be the set of pp-1-formulas such that $\phi_i^\alpha(N)\subseteq \elsoc^\alpha(N)$. All finite subsets of the system $\phi_i^\alpha(x-g_\alpha)$ where $\alpha<\lambda$ and $i\in I_\alpha$ have a solution in $\Hom_S(N,E(S/\mfrak{m}))$, namely $x=g_\beta$ where $\beta$ is the maximal ordinal occurring as a superscript of one of the $\phi^\alpha_i$s of the finite system. So, since $\Hom_S(N,E(S/\mfrak{m}))$ is pure-injective, there exists a $g_\lambda\in \Hom_S(N,E(S/\mfrak{m}))$ such that $\phi^\alpha_i(g_\lambda-g_\alpha)$ holds for all $\alpha <\lambda$ and $i\in I_\alpha$. Thus $(\nu-\overline{g_\lambda})|_{\elsoc^\lambda(N)}=0$.

Now to complete the proof, ordinal induction using 1. and 2. gives us that for every $\nu\in \Hom_k(N,k)$ there exists $f\in \Hom_S(N,E(S/\mfrak{m}))$ such that $(\nu-\overline{f})|_{\elsoc^\lambda(N)}=0$ where $\lambda$ is the elementary socle length of $N$. Thus $\nu-\overline{f}=0$ on all of $N$. Thus the local dual is isomorphic to the $k$-dual.
\end{proof}


\section{Ray tubes from generalised ray tubes}

Given a ray tube with translation quiver $Q(m;n_1,\ldots,n_m)$ we describe a generalised ray tube.

For each $j\in\N$, let $M_j:=S_0^0[j]\oplus\ldots S_m^0[j]$, let $\psi_j:M_j\rightarrow M_{j+1}$ be given by the matrix with $\mu_i^0[j]$ in the $i+1$st diagonal entry and zeros everywhere else and let $\phi_j$ be the matrix with compositions of $\lambda_i^{n_i}[j]\circ\ldots\circ\lambda_i^0[j]$ in the $(i+2, i+1)$ entry for $0\leq i <m$ and $\lambda_m^{n_i}[j]\circ\ldots\circ\lambda_m^0[j]$ in the $(m,m)$ entry, and zeros in all other entries.

For $1\leq l\leq\max{n_i}=n$, let $P^{l}=S_0^{l}[1]\oplus\ldots\oplus S_m^{l}[1]$ where $S_i^{l}[1]$ is taken to be zero if $l>n_i$ and let $\alpha^l$ be the diagonal embedding with entries $\lambda_i^l[1]$.

A module over a finite-dimensional algebra is \textbf{generic} if it is indecomposable and if its pp-$1$-lattice is finite-length (equivalently if it is indecomposable and finite-length over its endomorphism ring).

\begin{theorem}
The Ziegler closure of an Auslander-Reiten component $C$ over a finite-dimensional $k$-algebra with Auslander-Reiten quiver $Q(m;n_1,\ldots,n_m)$ consists of
\begin{enumerate}
\item the finite-dimensional indecomposable modules in $C$
\item  For each $0\leq i< m$ and $0\leq k\leq n_i$, an indecomposable $\Sigma$-pure-injective modules $S_i^k[\infty]:=\underrightarrow{\lim}S_i^k[j]$
\item For each $0\leq i< m$, an indecomposable pure-injective $\widehat{S_i^{0}}:=\underleftarrow{\lim}S_i^{0}[1+m(j-1)] $ which has acc on pp-definable subgroups
\item finitely many generic modules $G_1,\ldots,G_d$.
\end{enumerate}

\end{theorem}
\begin{proof}
Above we have shown how to construct a generalised ray tube from a ray tube. Thus all modules in the Ziegler closure of $C$ are either finite-dimensional or direct summands of $\underrightarrow{\lim}M_j$, $\underrightarrow{\lim}_jP_j^k$ for $1\leq k\leq n$, $\widehat{M}$ or $G$.

By construction $\underrightarrow{\lim}M_j=\oplus_{i=0}^{m-1}\underrightarrow{\lim}S_i^0[j]$ and $\underrightarrow{\lim}_jP_j^k=\oplus_{i=0}^{m-1}\underrightarrow{\lim}S_i^k[j]$. By \ref{directlimitsalongmuisind}, $\underrightarrow{\lim}S_i^k[j]$ is indecomposable for $0\leq i< m$ and $0\leq k\leq n_i$. By \ref{generalisedtubesclosure},  $\underrightarrow{\lim}M_j$ and $\underrightarrow{\lim}_jP_j^k$ for $1\leq k\leq n$ are $\Sigma$-pure-injective, thus $\underrightarrow{\lim}S_i^k[j]$ is $\Sigma$-pure-injective for $0\leq i< m$ and $0\leq k\leq n_i$.

By construction $\widehat{M}:=\oplus_{i=0}^{m-1}\underleftarrow{\lim}S_i^{0}[1+m(j-1)]=\oplus_{i=0}^{m-1}(\underrightarrow{\lim}S_i^{0}[1+m(j-1)])^*$ and $\widehat{M}$ has acc on pp-definable subgroups. Thus each $\underrightarrow{\lim}S_i^{0}[1+m(j-1)]$ has dcc on pp-definable subgroups and hence is pure-injective. By \ref{directlimitalonglambdadual}, $\underrightarrow{\lim}S_i^{0}[1+m(j-1)]$ is indecomposable. So, by  \ref{kdualisind}, $\underleftarrow{\lim}S_i^{0}[1+m(j-1)]=(\underrightarrow{\lim}S_i^{0}[1+m(j-1)])^*$ is indecomposable.

Since $G$ is finite-endolength, $G=G_1\oplus\ldots\oplus G_d$ where each $G_i$ is indecomposable and finite-endolength.
\end{proof}

In the following corollary we are less explicit in our description of the modules in the closure of a coray tube. This is because to do so would require the introduction of yet more notation.

\begin{cor}
The Ziegler closure of an Auslander-Reiten component $C$ over a finite-dimensional $k$-algebra with Auslander-Reiten quiver $Q^*(m;n_1,\ldots,n_m)$ consists of
\begin{enumerate}
\item the finite-dimensional indecomposable modules in $C$
\item  a direct limit along each ray
\item  an inverse limit along each coray
\item finitely many generic modules $G_1,\ldots,G_d$.
\end{enumerate}
\end{cor}
\begin{proof}
Since the Ziegler closure of a ray tube has m-dimension $2$ and hence the Ziegler closure of a coray tube also has m-dimension 2, all points in the Ziegler closure of $C$ are reflexive in the sense of Herzog \cite{Herzogduality} (See also \cite[page 271]{PSL}). Thus we know that there is a bijective correspondence $N\mapsto DN$ between points in the Ziegler closure of $C$ and points in the Ziegler closure of its $k$-dual such that $N\in\left(\phi/\psi\right)$ if and only if $DN\in\left(D\psi/D\phi\right)$. If the $k$-dual of an indecomposable pure-injective $N$ is indecomposable then $DN=N^*$ by \ref{kdualandduality}. On finite-dimensional points this bijection is just taking the $k$-dual.

The inverse limit along a coray is indecomposable and pure-injective by \ref{kdualisind}. The dual of the direct limit along a ray is indecomposable and has acc on pp-definable subgroups. Thus the direct limit along a ray is indecomposable and $\Sigma$-pure-injective.

Thus all modules on our list are indecomposable and pure-injective.

For any generic module $G$, $DG$ is also a generic module since its pp-$1$-lattice must be finite-length.

Comparing this list of modules with the list for the closure of a ray tube and taking $k$-duals of the appropriate modules, implies that the above list is complete.
\end{proof}

We now end the article with some open questions and directions for future research.

\begin{question}
Does every stable/ray/coray tube have a unique generic module in its closure?
\end{question}

The above question is open even for homogeneous tubes. There is an attempt at an example of a tube with more that one generic over a non-finite-dimensional algebra in \cite[15.1.11]{PSL} but this example doesn't seem to be a tube.

\begin{question}
Do there exist finite-dimensional algebras $A$ and $B$ with Auslander-Reiten components $C$ and $D$ such that $C$ is isomorphic to $D$ as a translation quiver but the Ziegler closure of $C$ is not homeomorphic to the Ziegler closure of $D$? More generally, can two finite-dimensional algebras have isomorphic auslander-reiten quivers but non-homeomorphic Ziegler spectra?
\end{question}

\begin{question}
Is the m-dimension of the closure of an Auslander-Reiten component of finite-dimensional algebra always finite?
\end{question}

\nocite{*}
\bibliographystyle{halpha}
\bibliography{onepointextensionsofvaluationdomains}
\end{document}